\documentclass[a4paper,11pt]{article}

\usepackage{xcolor}
\usepackage[pdfencoding=unicode, hidelinks]{hyperref}
\hypersetup{
	colorlinks,
	linkcolor={red},
	citecolor={blue},
}

\usepackage[english]{babel}
\usepackage[utf8]{inputenc}
\usepackage{amsthm}
\usepackage{amsmath}
\usepackage{amssymb}
\usepackage{mathrsfs}
\usepackage[mathscr]{eucal}
\usepackage{lmodern}
\usepackage{wasysym}
\usepackage{textcomp}
\usepackage{enumerate}
\usepackage{setspace}
\usepackage{latexsym}
\usepackage{graphicx}
\usepackage{mathtools}
\usepackage{tabularray}
\usepackage{geometry}
\usepackage{float}
\usepackage{emptypage}
\usepackage{commath}
\usepackage{multicol}
\usepackage{titlesec}
\usepackage[shortlabels]{enumitem}
\usepackage{verbatim}
\usepackage{cases}
\usepackage{cleveref}
\usepackage{orcidlink}
\usepackage{authblk}
\usepackage{mathbbol}

\usepackage{csquotes}
\usepackage[backend=biber,
    style=alphabetic,
    citestyle=alphabetic,
    sorting=nyt, 
    doi=false,
    url=false,
    eprint=true,
    isbn=false,
    maxnames=6,
    abbreviate=true,
    giveninits=true]{biblatex}
\AtEveryBibitem{%
  \clearfield{note}%
}

\theoremstyle{plain}
\newtheorem{thm}{Theorem}[section]
\theoremstyle{plain}

\theoremstyle{plain}

\theoremstyle{plain}

\theoremstyle{plain}
\newtheorem{prop}[thm]{Proposition}
\theoremstyle{definition}
\newtheorem{remark}[thm]{Remark}

\numberwithin{equation}{section}

\newcommand{\RR}{\mathbb{R}}

\newcommand{\dx}{\,\ensuremath{\mathrm{d}}x}
\newcommand{\ds}{\,\ensuremath{\mathrm{d}}s}
\newcommand{\dt}{\,\ensuremath{\mathrm{d}}t}

\renewcommand{\norm}[1]{\lVert#1\rVert}
\renewcommand{\phi}{\varphi}
\newcommand{\duality}[2]{\langle#1,#2\rangle}
\newcommand{\dualprod}[2]{\langle#1,#2\rangle_{\Vs}}
\newcommand{\intprod}[2]{(#1,#2)_{\Hs}}

\newcommand{\Ws}{W}
 
\newcommand{\Vs}{V}
\newcommand{\Hs}{H}

\newcommand{\nnu}{\boldsymbol{\nu}}

\newcommand{\hbeta}{\widehat{\beta}}

\newcommand{\hpi}{\widehat{\pi}}
\newcommand{\hf}{\widehat{f}}

\newcommand{\lambdap}{\lambda_P}
\newcommand{\lambdaa}{\lambda_A}
\newcommand{\lambdae}{\lambda_E}
\newcommand{\lambdac}{\lambda_C}
\newcommand{\lambdab}{\lambda_B}
\newcommand{\lambdad}{\lambda_D}

\DeclareSymbolFontAlphabet{\mathbb}{AMSb}
\DeclareSymbolFontAlphabet{\mathbbm}{bbold}
\newcommand{\h}{\mathbbm{h}}
\renewcommand{\k}{\mathbbm{k}}

\newcommand{\difft}{\frac{\text{d}}{\text{d}t}}

\newcommand{\into}{\int_{\Omega}}
\newcommand{\intTo}{\int_0^T \!\!\!\!\into}
\newcommand{\inttTo}{\int_t^T \!\!\!\!\into}

\newcommand{\intto}{\int_0^t \!\!\!\into}
\newcommand{\intt}{\int_0^t}
\newcommand{\inttT}{\int_t^T}

\newcommand{\U}{\mathcal{U}}
\newcommand{\Uad}{\U_{\text{ad}}}
\newcommand{\Um}{\U_{M}}
\renewcommand{\S}{\mathcal{S}}

\newcommand{\zetah}{\zeta^{h}}
\newcommand{\xih}{\xi^{h}}
\newcommand{\etah}{\eta^{h}}
\newcommand{\rhoh}{\rho^{h}}
\newcommand{\thetah}{\theta^{h}}
\newcommand{\phih}{\phi^{h}}
\newcommand{\muh}{\mu^{h}}
\newcommand{\sigmah}{\sigma^{h}}
\newcommand{\Thetah}{\Theta^{h}}
\newcommand{\Phih}{\Phi^{h}}
\newcommand{\Upsilonh}{\Upsilon^{h}}
\newcommand{\Pih}{\Pi^{h}}
\newcommand{\q}{\mathrm{q}}
\renewcommand{\r}{\mathrm{r}}
\newcommand{\z}{\mathrm{z}}
\newcommand{\p}{\mathrm{p}}

\newcommand{\J}{\mathrm{J}}
\newcommand{\JJ}{\mathcal{J}}

\newcommand{\lhs}{left-hand side}
\newcommand{\rhs}{right-hand side}

\crefname{lemma}{lemma}{lemmas}
\Crefname{lemma}{Lemma}{Lemmas}
\crefname{prop}{proposition}{proposition}
\Crefname{prop}{Proposition}{Propositions}
\crefname{cor}{corollary}{corollaries}
\Crefname{cor}{Corollary}{Corollaries}
\crefname{remark}{remark}{remarks}
\Crefname{remark}{Remark}{Remarks}
\crefname{thm}{theorem}{theorems}
\Crefname{thm}{Theorem}{Theorems}
\Crefname{section}{Section}{Sections}

\begin{document}

\begin{center}
		
    {\Large \bf Optimal control for a fourth-order nonisothermal\\[2mm]
    tumor growth model of Caginalp type}

    \vskip0.5cm
    
    {\large\textsc{Giulia Cavalleri\orcidlink{0009-0006-4154-9659}$^1$}} \\
    {\normalsize e-mail: \texttt{cavalleri@wias-berlin.de}} \\
    \vskip0.35cm

    {\large\textsc{Pierluigi Colli\orcidlink{0000-0002-7921-5041}$^{2,3}$}} \\
    {\normalsize e-mail: \texttt{pierluigi.colli@unipv.it}} \\
    \vskip0.35cm

    {\large\textsc{Elisabetta Rocca\orcidlink{0000-0002-9930-907X}$^{2,3}$}} \\
    {\normalsize e-mail: \texttt{elisabetta.rocca@unipv.it}} \\
    \vskip0.35cm

    {\footnotesize $^1$Weierstrass Institute,  Anton-Wilhelm-Amo-Str 39, 10117 Berlin, Germany }
    \vskip0.1cm
  
    {\footnotesize $^2$Department of Mathematics ``F. Casorati'', University of Pavia, 27100 Pavia, Italy}
    \vskip0.1cm
    
    {\footnotesize$^3$Research Associate at the IMATI -- C.N.R. Pavia, via Ferrata 5, 27100 Pavia, Italy}
    \vskip0.5cm
		
\end{center}

\begin{abstract}
    \noindent We study a distributed optimal control problem for a nonisothermal Caginalp-type phase-field model that describes tumour growth under thermal therapy.  The PDE system couples a possibly viscous Cahn--Hilliard equation, governing the evolution of the healthy and tumor phases, with an equation for the heat balance, and a reaction-diffusion equation for the nutrient concentration. Chemotaxis and active transport effects are taken into account, and hyperthermia appears as a control variable. We introduce a suitable tracking-type cost functional and show the existence of optimal controls. Then, we analyse the differentiability of the control-to-state operator and establish necessary first-order conditions expressed through a variational inequality involving the adjoint state variables.  
    \vskip3mm
    \noindent {\bf Key words:} tumor growth model, Cahn--Hilliard system, nonisothermal model, optimal control, adjoint system,  necessary optimality conditions.
    \vskip3mm
    \noindent {\bf AMS (MOS) Subject Classification:} 
    35Q93, 
    49J20,  
    49K20, 
    35Q92, 
    92C50. 
\end{abstract}

\section{Introduction}
In this paper, we study an optimal control problem for a nonisothermal tumor growth model, which was introduced and proved to be well-posed in our previous work \cite{Cavalleri_Colli_Rocca_25}. Let $T>0$ be the final time and $\Omega \subseteq \RR^d$ be a sufficiently regular spatial domain in dimension $d=2,3$, having outward unit normal vector on $\partial \Omega$ given by $\nnu$. The domain $\Omega $ represents the portion of the body under investigation. The main variables whose evolution is studied are:
\begin{itemize}
    \item the relative temperature $\theta$ in $\Omega$;
    \item the phase-field variable $\phi$, describing the difference in volume fraction between tumor and healthy cells (in particular, the regions $\{\phi = 1\}$ and $\{\phi = -1\}$ correspond to tumor and healthy tissue, respectively), together with its associated chemical potential $\mu$;
    \item the concentration $\sigma$ of a nutrient (e.g., oxygen or glucose) dissolved in the blood.
\end{itemize}
For an assigned control $u$, representing a thermal therapy administered to the patient in the region under consideration, the corresponding state $(\theta, \phi, \mu, \sigma)$ is the unique solution of the PDE system
\begin{subequations}\label{eq:system}
    \begin{align}
        & \partial_t(\theta +\ell\phi) - \Delta \theta = u,\label{eq:temp}\\
        & \partial_t\phi - \Delta \mu = (\lambdap \sigma - \lambdaa - \lambdae \theta)\h(\phi),\label{eq:ch1}\\
        & \mu = \tau \partial_t \phi - \Delta \phi + \beta(\phi) + \pi(\phi) - \chi \sigma -\Lambda \theta, \label{eq:ch2}\\
        & \partial_t \sigma - \Delta (\sigma - \chi \phi) = - \lambdac \sigma \h(\phi) + \lambdab (\sigma_B - \sigma) - \lambdad \sigma \k(\theta), \label{eq:nutrient}
    \end{align}
\end{subequations}
posed in the parabolic cylinder $Q \coloneqq \Omega \times (0,T)$ and supplemented with no-flux boundary conditions
\begin{equation}\label{eq:boundary_cond}
    \partial_{\nnu} \theta = \partial_{\nnu} \phi = \partial_{\nnu} \mu = \partial_{\nnu} \sigma = 0
\end{equation}
on $\Sigma \coloneqq \partial \Omega \times (0,T)$, reflecting the idealised assumption that the portion $\Omega$ of the body is isolated from its surroundings, and with the initial conditions
\begin{equation}\label{eq:initial_datum}
     \theta(0) = \theta_0, \quad \phi(0) = \phi_0, \quad \sigma(0) = \sigma_0
\end{equation}
in $\Omega$. Here, the well-known (possibly viscous) Cahn--Hilliard 
reaction-diffusion subsystem \eqref{eq:ch1}--\eqref{eq:nutrient} governs the evolution of the tumor phase parameter $\phi$ and the nutrient 
concentration $\sigma$, and is coupled with the temperature equation \eqref{eq:temp}. In particular, the coupling between the relative 
temperature $\theta$ and the tumor variable $\phi$ is of Caginalp type (cf.~\cite{Caginalp_88}).
For the reader's convenience, we briefly recall the main features of the model, referring to the introduction of \cite{Cavalleri_Colli_Rocca_25} for a more comprehensive discussion.
The positive constants $\ell$ and $\Lambda$ are related to the latent heat of the tissue. 
We suppose that tumor cells proliferate by consumption of the nutrients and die because of apoptosis, with $\lambda_p$ and $\lambda_a$ being, respectively, the proliferation and apoptosis constant rates. Moreover, following \cite{Cavalleri_Colli_Rocca_25}, we assume that $\theta$ has cytotoxic effects proportional to the temperature, which we incorporate in the right-hand side of equation \eqref{eq:ch1} through the term $-\lambdae \theta$, where $\lambdae$ is a fixed positive parameter. 
The function $\h$, multiplying the mass source term in equation \eqref{eq:ch1}, guarantees that the phenomena
we have just described are proportional to the tumor cells available in a certain area. For
example, $\h$ may be a monotone increasing function which is 0 where $\{\phi = -1\}$ and 1 where $\{\phi = 1\}$. The nonnegative constant $\chi$ denotes a transport coefficient accounting for chemotaxis and active transport effects.
The term $\beta(\phi) + \pi(\phi)$ represents the derivative of a double-well potential. We allow for a fairly general polynomial growth, so that, in particular, the classical choice $\beta(\phi) + \pi(\phi) = \phi^3 - \phi$ is admissible. In contrast, singular potentials of logarithmic type as well as double-obstacle potentials fall outside the frame of the present analysis.
The parameter $\tau \geq 0$ may be either strictly positive or zero. When $\tau > 0$, the term~$\tau \partial_t \phi$ in \eqref{eq:ch2} acts as a viscous regularisation in the phase-field dynamics. The case $\tau = 0$ instead corresponds to the classical (non-viscous) Cahn--Hilliard system. Both regimes are addressed in our analysis. From a mathematical viewpoint, the presence of the viscous term improves the regularity of solutions, which may be advantageous in view of possible extensions to singular potentials.
Concerning the right-hand side of \eqref{eq:nutrient}, the term $- \lambdac \sigma \h(\phi)$ models nutrient consumption by tumor cells, with a higher consumption rate in regions where the tumor cell density is larger. The term $\lambdab (\sigma_B - \sigma)$ describes the supply of nutrients from the pre-existing vasculature. Finally, the function $\k$ accounts for the influence of temperature on nutrient uptake, for instance, through vasodilation effects.

\paragraph{The control problem.} Within this framework, a well-chosen heat source $u$ could drive the system toward a tumor reduction or, more generally, toward a desired evolution or final state. We introduce some target functions $\theta_Q, \phi_Q$, and $\theta_{\Omega}, \phi_\Omega$ representing, respectively, the desired evolution of the temperature and of the phase parameter, and the desired temperature and tumor proportion at the final time $T$. For example, the target $\phi_\Omega$ could be a suitable configuration for surgery. Then, we define the cost functional 
\begin{equation}\label{def:J}
    \begin{split}
        \JJ(\theta, \phi, u) \coloneqq \ 
        &\frac{b_1}{2} \int_Q |\theta - \theta_Q|^2 \dx \dt + \frac{b_2}{2} \into |\theta(T) - \theta_{\Omega}|^2 \dx\\ 
        & + \frac{b_3}{2} \int_Q |\phi - \phi_Q|^2 \dx \dt + \frac{b_4}{2} \into |\phi(T) - \phi_{\Omega}|^2 \dx\\ 
        & + \frac{b_5}{2} \int_Q |u|^2 \dx \dt .
    \end{split}
\end{equation}
The first four terms in \eqref{def:J} are of standard tracking-type, while the last regularising term prevents the administration of an excessive amount of therapy, thereby avoiding damage to other vital organs. 
Firstly, we aim at minimising the cost functional $\JJ$ subject to the PDE system \eqref{eq:system}--\eqref{eq:initial_datum} and to the following admissible control set
\begin{equation}
    \Uad \coloneqq \left\{ u \in \U \, | \, u_{\textrm{min}} \leq u \leq u_{\textrm{max}} \text{ a.e. in } Q\right\},
\end{equation}
where $\U \coloneqq L^{\infty}(Q)$, and $u_{\textrm{min}}, u_{\textrm{max}}$ are assigned elements of $\U$ satisfying $u_{\textrm{min}} \leq u_{\textrm{max}}$ almost everywhere in $Q$. Notice that $\Uad$ is a nonempty, closed, bounded and convex subset of $\U$. In particular, there exists a constant $M>0$ such that $\Uad \subseteq \Um$, where the open set $\Um$ is given~by
\begin{equation}
\label{betti1}
    \Um \coloneqq \left\{ u \in \U \,\middle|\, \norm{u}_{L^{\infty}(Q)} < M \right\}.
\end{equation}
Rephrasing the optimal control problem more compactly, it reads as follows:
\begin{equation}\label{eq:optimal_control}
        \min_{u\in \Uad} \JJ(\theta,\phi,u).
\end{equation}
The subsequent, more challenging, goal is to characterise the optimal controls, i.e., the minimisers, through their necessary first-order optimality conditions.

\paragraph{Related literature.} Let us recall some of the literature related to tumor growth models based on Cahn--Hilliard type dynamics, which is expanding rapidly, reflecting the strong interest it has attracted from both pure and applied mathematicians. In most of the previous papers on the subject, the isothermal case was considered,  assuming, indeed, constant temperature.
The reader can refer, e.g., to \cite{Frigeri_Grasselli_Rocca_15, Garcke_Lam_Sitka_Styles_16, Garcke_Lam_17, Ebenbeck_Garcke_Nurnberg_21} and reference therein, where the nutrient dynamic is taken into account by means of a reaction-diffusion equation analogous to our \eqref{eq:nutrient}. Moreover, various works also incorporate the evolution of other relevant quantities like elastic effects, and fluid flows (see, e.g., \cite{Colli_Gomez_Lorenzo_etal_20, Cavalleri_Colli_Miranville_Rocca_2026, Garcke_Lam_Signori_21, Garcke_Kovacs_Trautwein_22, Garcke_Lam_Sitka_Styles_16, Knopf_Signori_22} and the references therein). 
On the other hand, nonisothermal phase-field models have also been widely investigated in the literature, although not in the context of tumor growth. In this regard, we mention, e.g., \cite{DeAnna_etal_24, Miranville_Schimperna_05,heinemann_kraus_rocca_rossi_2017, Lasarzik_21, Colli_Gilardi_Signori_Sprekels_24}. The recent paper \cite{gatti-ipocoana-miranville-25} is devoted to the analysis of a nonisothermal tumor growth model in which the phase evolution is governed by an Allen--Cahn dynamics.
However, to the best of our knowledge, only the recent contributions \cite{Cavalleri_Colli_Rocca_25, Ipocoana_22} incorporate temperature effects into a tumor model of Cahn--Hilliard type.
In the paper~\cite{Ipocoana_22}, the PDE system is rigorously derived in terms of an entropy balance, but only the existence of weak solutions for the associated initial-boundary value problem is proved. Hence, the main objective of our previous contribution \cite{Cavalleri_Colli_Rocca_25} has been to address a different system going beyond proving the existence of weak solutions and establishing higher regularity estimates on the solution as well as continuous dependence on the data, which pave the way to the analysis performed in the present paper on the associated optimal control problem.

Regarding the problem of distributed and boundary optimal control for diffuse interface tumor growth problems coupling Cahn--Hilliard type dynamics for the tumor with reaction-diffusion equations for the nutrient, we can quote, e.g., the papers  \cite{Cavaterra_Rocca_Wu_21, Colli_Gilardi_Rocca_Sprekels_17, Garcke_Lam_Rocca_18}. Control strategies incorporating chemotaxis, active transport, variable
mobilities, and Keller--Segel dynamics were developed in \cite{Abatangelo_Cavaterra_Grasselli_Wu_24, Agosti_Signori_24, Colli_Gilardi_Signori_Sprekels_25, Colli_Signori_Sprekels_21, Ebenbeck_Knopf_20, Ebenbeck_Knopf_19}. Further contributions
included optimal control for nonlocal phase-field type tumor growth models  (where the Cahn--Hilliard equation is replaced by its nonlocal variant), e.g., in \cite{Fornoni_24a, Fornoni_24b}. 
Concerning optimal control for non-isothermal phase-field models of Caginalp--Cahn--Hilliard type, there are several contributions in the literature dealing with distributed and boundary control problems. The reader can refer, e.g., to \cite{Colli_Gilardi_Signori_Sprekels_23a, Colli_Gilardi_Signori_Sprekels_23b} and reference therein.
However, as far as we are aware, this is the first contribution where an optimal control problem is introduced for a phase-field system where temperature, phase, and nutrient dynamics are coupled.

\paragraph{Plan of the paper.} The paper is organized as follows.
In Section~\ref{sec:notation}, we introduce the notation and collect some standard inequalities and auxiliary results that will be used repeatedly throughout the paper. In Section~\ref{sec:hyp}, we state the assumptions on the problem data.  In Section~\ref{sec:optim}, we recall the results of \cite{Cavalleri_Colli_Rocca_25} related to well-posedness and regularity of the state system \eqref{eq:system}--\eqref{eq:initial_datum} and prove the existence of optimal control. In Sections~\ref{sec:lin}, \ref{sec:diff}, we prove well-posedness of the associated linearised system and the Fr\'echet differentiability of the control-to-state mapping. Finally, in the last Section~\ref{sec:first} we prove first-order necessary optimality conditions in terms of the adjoint system.

\section{Notation and preliminaries}
\label{sec:notation}
Throughout the paper, we consider a bounded domain $\Omega$  of class $C^2$ in $\RR^d$ with $d = 2,3$, and we fix a final time $T > 0$.
We recall that $Q=\Omega \times (0,T) $ and $\Sigma=\partial \Omega \times (0,T) $.
If $(X,\norm{\cdot}_X)$ is a Banach space, we denote by $(X^*,\norm{\cdot}_{X^*})$ its (topological) dual. We write $L^p(\Omega)$ and $W^{k,p}(\Omega)$ for the usual Lebesgue and Sobolev spaces defined on $\Omega$. In the Hilbert case $p=2$, we adopt the standard notation $H^k(\Omega) := W^{k,2}(\Omega)$.
For Bochner spaces of the form $W^{k,p}(0,T;X)$, the norm will be denoted by $\norm{\cdot}_{W^{k,p}(X)}$, omitting the explicit reference to the time interval $(0,T)$ whenever no confusion arises. Should a different final time be considered, it will be specified explicitly.
We denote by $C^0([0,T];X)$ the space of continuous functions from $[0,T]$ into $X$, and by $C^0_w([0,T];X)$ the space of functions $v \in L^\infty(0,T;X)$ that are weakly continuous on $[0,T]$ with values in $X$.
For later use, we introduce the spaces
\begin{equation*}
    W := \{ v \in H^2(\Omega) \,:\, \partial_{\nnu} v = 0 \text{ on } \partial\Omega \}, 
    \qquad V := H^1(\Omega), 
    \qquad H := L^2(\Omega).
\end{equation*}
The inner product in $H$ is denoted by $\intprod{\cdot}{\cdot}$, while $\dualprod{\cdot}{\cdot}$ stands for the duality pairing between $V^*$ and $V$. As customary, we identify $H$ with its dual $H^*$ and regard it as a subspace of $V^*$ via
\begin{equation*}
    \dualprod{w}{v} = \intprod{w}{v}
\end{equation*}
for all $w \in H$ and $v \in V$.
We now collect some useful tools that will be employed in the sequel. Besides H\"older's inequality, we shall frequently use Young's, Sobolev's, and Poincar\'e's inequalities, as well as some estimates stemming from elliptic regularity theory and from the compact embeddings $V \hookrightarrow L^p(\Omega)$ for every $p \in [1,6)$ and $W \hookrightarrow V$.
In particular, the following inequalities hold:
    \begin{align}
        &ab \le \delta a^2 + \frac{1}{4\delta} b^2 \quad \text{for every } a,b \in \mathbb{R}, \label{pier1}\\
        &\|v\|_{L^p(\Omega)} \le C_\Omega \|v\|_{\Vs} \quad \text{for every } p \in [1,6], \label{pier2}\\
        &\|v\|_{\Vs} \le C_\Omega \big( \|\nabla v\|_{\Hs} + |\overline{v}| \big), \label{pier3}\\
        &\|w\|_{\Ws} \le C_\Omega \big( \|\Delta w\|_{\Hs} + \|w\|_{\Hs} \big), \label{pier4}\\
        &\|v\|_{L^p(\Omega)} \le \delta \|\nabla v\|_{\Hs} + C_{\Omega,p,\delta} \|v\|_{\Hs}\quad \text{for every } p \in [1,6), \label{pier5}\\
        &\|w\|_{\Vs} \le \delta \|\Delta w\|_{\Hs} + C_{\Omega,\delta} \|w\|_{\Hs}, \label{pier6}\\
        &\|w\|_{L^\infty(\Omega)} \le \delta \|\Delta w\|_{\Hs} + C_{\Omega,\delta} \|w\|_{\Hs}, \label{pier7}
    \end{align}
for every $v \in \Vs$, $w \in \Ws$, and $\delta > 0$.
The constants $C_\Omega$ depend only on $\Omega$, while $C_{\Omega,\delta}$ may also depend on $\delta$. 
Here, $\overline{v}$ denotes the mean value of $v$. 
More generally, let us define the generalized mean value $\overline{v}$ of a generic element $v \in \Vs^*$ by setting
\begin{equation*}
    \overline{v} := \frac{1}{|\Omega|} \langle v, 1 \rangle_{\Vs}, \label{pier8}
\end{equation*}
where $1$ denotes the constant function equal to $1$ in $\Omega$.  Note that  \eqref{pier8} makes sense since $1 \in \Vs$, and that $\overline{v}$ reduces to the usual mean value when $v \in \Hs$. The same notation $\overline{v}$ is also used when $v$ is time-dependent.
The notation $*_t$ denotes the convolution with respect to time, i.e., 
\begin{equation*}
   ( u*_t v ) (t) \coloneqq \intt u(t-s) v(s) \ds , \quad t\in (0,T). \label{pier9} 
\end{equation*}
Finally, we point out that the symbol $C$ denotes a generic positive constant depending only on the structural data of the problem; its value may change from line to line. If it is necessary to emphasise the dependence on a parameter, say $\delta>0$, we write $C_\delta$. Specific constants to which we refer explicitly are denoted in different ways.

\section{Hypotheses}
\label{sec:hyp} 

We now enlist the assumptions on the assigned parameters, the nonlinear terms, and the prescribed data that will be in force throughout the paper.

\begin{enumerate}[\rm{(H\arabic*)}]
\item \label{hyp:constants}
We assume that 
\begin{align}
    & \ell, \Lambda, \chi \text{ are positive constants},\\
    & \tau \text{ and }\lambdap, \lambdaa, \lambdae, \lambdac, \lambdab, \lambdad  \text{ are real nonnegative constants}.
\end{align}
\item \label{hyp:given_function}
The assigned function $\sigma_B$ enjoys the regularity
\begin{align}
    & \sigma_B \in L^2(Q).
\end{align} 
\item \label{hyp:nonlinearities}
Regarding the nonlinearities $\h$ and $\k$, we require that
\begin{align}
    & \h, \k \in C^{0,1}(\RR) \cap C^2(\RR), \text{ and}\\
    &0 \leq \h \leq \h^*,\ 0 \leq \k \leq \k^*,
\end{align}
for two nonnegative constants $\h^*,\ \k^*$.
\item \label{hyp:beta_pi}
We consider a potential $\hf = \hbeta + \hpi$ split into the sum of a convex part and a nonconvex perturbation. We denote its derivative as $f \coloneqq \beta + \pi$ where $\beta \coloneqq \hbeta'$ and $\pi \coloneqq \hpi'$. We suppose:
\begin{align}
     & \hbeta, \,  \hpi \in C^2(\RR), \label{hyp:hatf_reg} \\
     & \hbeta \text{ is convex and nonnegative,}\label{hyp:hatbeta}
\end{align}
and that the following growth conditions hold:
\begin{align}
    & |\beta(r)| \leq C_{\beta} (\hbeta(r) + 1),
    \label{hyp:growth_beta}\\[1mm]
    & |\pi'(r)| \leq C_{\pi}
    \label{hyp:pi_lip}
\end{align}
for all $r \in \RR$, where $C_{\beta}$, $C_{\pi}$ are given nonnegative constants. 
\end{enumerate}

\begin{remark}
    While the decomposition of the potential $\hf$, together with the assumptions \eqref{hyp:hatbeta}--\eqref{hyp:pi_lip} on $\hbeta$ and $\hpi$, plays a crucial role in the proof of well-posedness of the state system in \cite{Cavalleri_Colli_Rocca_25}, in the present paper we systematically adopt the more compact notation $\hf$ and $f$.
    The reason is that, once existence and uniqueness of the solution have been established and the state variable $\phi$ is known to belong to $L^{\infty}(Q)$ (cf. \Cref{thm:wellposedness}), the local Lipschitz continuity of $f$, which follows from its regularity assumption \eqref{hyp:hatf_reg}, is sufficient to justify all subsequent estimates and calculations. Therefore, the finer splitting of the potential is no longer needed at the level of the arguments developed here.
\end{remark}
\begin{enumerate}[\rm{(H\arabic*)},resume]
\item \label{hyp:init_cond}
The initial data have the following regularity:
\begin{equation}
    \theta_0 \in \Vs \cap L^{\infty}(\Omega),\quad \phi_0 \in \Ws \cap H^3(\Omega), \quad \sigma_0 \in \Vs.
\end{equation}
\end{enumerate}

\noindent The assumptions for the components of the cost functional are the following.
\begin{enumerate}[\rm{(H\arabic*)},resume]
\item \label{hyp:cost-fun}
The cofficients $b_1, b_2, b_3, b_4$ are real nonnegative, whereas $b_5$ is positive. In addition, we suppose that 
\begin{equation}
    \theta_Q , \phi_Q \in  L^{2}(Q),\quad \theta_{\Omega} , \phi_{\Omega} \in \Vs .
\end{equation}
\end{enumerate}%

\section{The optimal control problem}
\label{sec:optim}

The well-posedness of the state system \eqref{eq:system}--\eqref{eq:initial_datum} has been previously studied in the contribution \cite{Cavalleri_Colli_Rocca_25}, to which we refer for further details on the matter. For the reader's convenience, the following theorem collects the main results in~\cite{Cavalleri_Colli_Rocca_25}. 
 \begin{thm}\label{thm:wellposedness}
    Assume that hypotheses \ref{hyp:constants}--\ref{hyp:init_cond} hold. Then, for every $u \in \Um$, the PDE system~\eqref{eq:system}--\eqref{eq:initial_datum} has a unique strong solution $(\theta, \phi, \mu, \sigma)$
    belonging to the following spaces
    \begin{gather}
        \theta \in  H^1(0,T;\Hs) \cap C^0([0,T];\Vs)  \cap L^2(0,T;\Ws) \cap L^{\infty}(Q),\label{pcol1}\\  
        \phi \in   W^{1,\infty}(0,T;\Vs') \cap  H^1(0,T;\Vs) \cap  L^{\infty}(0,T;\Ws) ,\label{pcol2}\\
        \mu \in L^{\infty}(0,T;\Vs) \cap  L^2(0,T;\Ws),\label{pcol3}\\
        \sigma \in  H^1(0,T;\Hs) \cap  C^0([0,T];\Vs) \cap  L^2(0,T;\Ws), \label{pcol14}
    \end{gather}
    which satisfies the estimate
    \begin{equation}\label{eq:solution_estimate}
        \begin{split}
            &\norm{\theta}_{H^1(\Hs)\cap L^\infty(\Vs) \cap L^2(\Ws) \cap L^{\infty}(Q)} + \norm{\phi}_{W^{1,\infty}(\Vs')\cap H^1(\Vs) \cap  L^{\infty}(\Ws)}\\
            &\quad+ \norm{\mu}_{L^{\infty}(\Vs) \cap  L^2(\Ws)} + \norm{\sigma}_{H^1(\Hs) \cap L^{\infty}(\Vs) \cap L^2(\Ws)} \leq C_M.
        \end{split}
    \end{equation}
    Moreover, for every couple of controls $\{u_i\}_{i=1,2} \subseteq \Um$ the corresponding strong solutions $\{(\theta_i, \phi_i, \mu_i, \sigma_i)\}_{i=1,2}$ satisfy the following continuous dependence inequality
    \begin{equation}
        \begin{split}
            &\norm{\theta_1-\theta_2}_{L^{\infty}(\Hs) \cap L^2(\Vs)} + \norm{\phi_1-\phi_2}_{L^{\infty}(\Hs)\cap L^2(\Ws)} + \norm{\tau^{1/2}(\phi_1-\phi_2)}_{L^{\infty}(\Vs)}\\
            &\quad \quad + \norm{\mu_1 - \mu_2}_{L^2(\Hs)} + \norm{\tau^{1/2}(\nabla\mu_1 - \nabla \mu_2)}_{L^2(\Hs)} + \norm{\sigma_1-\sigma_2}_{L^{\infty}(\Hs)\cap L^2(\Vs)}\\
            &\quad \leq C_{M} \norm{u_1-u_2}_{L^2(\Hs)}.
        \end{split}
        \label{eq:cont_dep}
    \end{equation}
 \end{thm}

\begin{proof}
For the existence and regularity of a solution $(\theta, \phi, \mu, \sigma)$ satisfying \eqref{pcol1}--\eqref{pcol14}
and \eqref{eq:solution_estimate} we refer to \cite[][Theorems~3.5 and 3.6]{Cavalleri_Colli_Rocca_25}. Part of the continuous dependence estimate~\eqref{eq:cont_dep} has been established in \cite[][Theorem 3.8]{Cavalleri_Colli_Rocca_25}, namely the following  
\begin{equation}
        \begin{split}
            &\norm{\theta_1-\theta_2}_{L^{\infty}(\Hs) \cap L^2(\Vs)} + \norm{\phi_1-\phi_2}_{L^{\infty}(\Hs)\cap L^2(\Ws)} + \norm{\tau^{1/2}(\phi_1-\phi_2)}_{L^{\infty}(\Vs)}\\
            &\quad \quad  + \norm{\sigma_1-\sigma_2}_{L^{\infty}(\Hs)\cap L^2(\Vs)} \leq C_{M} \norm{u_1-u_2}_{L^2(\Hs)}.
        \end{split}
        \label{eq:con_dep_pier}
    \end{equation}
    We now prove the full \eqref{eq:cont_dep} for the couple of controls $\{u_i\}_{i=1,2} \subseteq \Um$ and the corresponding states $\{(\theta_i, \phi_i, \mu_i, \sigma_i)\}_{i=1,2}$. For the sake of brevity, we introduce the notation
    \begin{equation*}
        (\theta, \phi, \mu, \sigma) \coloneqq (\theta_1 - \theta_2, \phi_1 - \phi_2, \mu_1 - \mu_2, \sigma_1 - \sigma_2).
    \end{equation*}
    For each solution, we sum the equation \eqref{eq:ch1} multiplied by $\tau \geq 0$ and the equation \eqref{eq:ch2}. Then, we take the difference of the resulting equations, test it by $\mu$, and integrate over the time interval $(0,T)$. We obtain:
    \begin{equation*}
        \begin{split}
            &\tau\intTo |\nabla \mu|^2 \dx \dt + \intTo |\mu|^2 \dx \dt\\
            &\quad = \tau \intTo \left[(\lambdap \sigma - \lambdae \theta)\h(\phi_1) + (\lambdap \sigma_2  - \lambdaa - \lambdae \theta_2)\big(\h(\phi_1) - \h(\phi_2)\big)\right]\mu \dx \dt\\
            & \quad \quad + \intTo \left[-\Delta \phi + \big(f(\phi_1)-f(\phi_2)\big) - \chi\sigma - \Lambda \theta \right] \mu \dx \dt.
        \end{split}
    \end{equation*}
    We bind the right-hand side from above, exploiting boundedness and Lipschitz continuity of $\h$ from hypothesis \ref{hyp:nonlinearities}, local Lipschitz continuity of $f$ from \ref{hyp:beta_pi}, as well as uniform boundedness of $\theta_1, \theta_2$ and $\phi_1, \phi_2$ from \Cref{thm:wellposedness}. Thus, we make use of the H\"older inequality---also recalling that $\sigma_2$ is uniformly bounded in $L^{\infty}(0,T;\Vs) \hookrightarrow L^{\infty}(0,T;L^4(\Omega))$--- followed by the Young inequality:
        \begin{equation*}
        \begin{split}
            &\tau\intTo |\nabla \mu|^2 \dx \dt + \intTo |\mu|^2 \dx \dt\\
            & \quad  \leq C \intTo \left[(\tau + 1) \big(|\sigma| + |\theta|\big) + |-\Delta \phi| + \big( |\sigma_2| + 1 \big)|\phi|\right] |\mu| \dx \dt\\
            &\quad \leq C \int_0^T\left[ (\tau + 1) \big(\norm{\sigma}_{\Hs} + \norm{\theta}_{\Hs}\big) + \norm{-\Delta \phi}_{\Hs} + \big(\norm{\sigma_2}_{L^4(\Omega)} + 1\big)\norm{\phi}_{L^4(\Omega)}\right]\norm{\mu}_{\Hs} \dt\\
            & \quad \leq \frac{1}{2} \int_0^T \norm{\mu}_{\Hs}^2 \dt + C \int_0^T \left[(\tau^2 + 1)\big(\norm{\sigma}_{\Hs}^2 + \norm{\theta}_{\Hs}^2 \big) + \norm{\phi}_{\Ws}^2 \right] \dt.
        \end{split}
    \end{equation*}
    At this point, in view of \eqref{eq:con_dep_pier} we deduce that
    \begin{equation*}
       \begin{split}
           \tau \norm{\nabla \mu}_{L^2(\Hs)}^2 + \norm{\mu}_{L^2(\Hs)}^2 &\leq  C(\tau^2 + 1) \big(\norm{\sigma}_{L^2(\Hs)}^2 + \norm{\theta}_{L^2(\Hs)}^2\big) + C \norm{\phi}_{L^2(\Ws)}^2\\
           & \leq C (\tau^2 + 1) \norm{u_1 -u_2}_{L^2(\Hs)}^2 + C \norm{u_1 -u_2}_{L^2(\Hs)}^2
       \end{split}
    \end{equation*}
and conclude the proof.
\end{proof}
 
Note that ~\Cref{thm:wellposedness} allows us to define the so-called control-to-state or solution operator. More precisely, we introduce the state-space 
 \begin{equation*}
     \begin{split}
         \mathcal{X} \coloneqq &\left[H^1(0,T;\Hs) \cap C^0([0,T];\Vs)  \cap L^2(0,T;\Ws) \cap L^{\infty}(Q)\right]\\
         &\times \left[W^{1,\infty}(0,T;\Vs') \cap  H^1(0,T;\Vs) \cap  L^{\infty}(0,T;\Ws) \right] \\
         &\times \left[ L^{\infty}(0,T;\Vs) \cap  L^2(0,T;\Ws)\right] \times \left[ H^1(0,T;\Hs) \cap  C^0([0,T];\Vs) \cap  L^2(0,T;\Ws)\right],
     \end{split}
 \end{equation*}
and observe that the solution operator
 \begin{equation*}
     \S \colon \Um \to \mathcal{X}, \qquad u \mapsto (\theta,\phi\,\mu,\sigma),
 \end{equation*}
 which maps every control $u \in \Um$ to the corresponding strong solution of the state system $(\theta,\phi\,\mu,\sigma) \in \mathcal{X}$, is well-defined and bounded.
 Moreover, thanks to the continuous embedding 
 \begin{equation*}
     \begin{split}
          \mathcal{X} \subseteq \mathcal{Y} \coloneqq &\left[C^0([0,T]; \Hs) \cap L^2(0,T;\Vs)\right] \times \left[C^0([0,T]; \Hs) \cap L^2(0,T;\Ws)\right] \\
         &\times L^2(0,T;\Hs)\times \left[C^0([0,T]; \Hs) \cap L^2(0,T;\Vs)\right],
     \end{split}
 \end{equation*}
 and to the continuous dependence estimate \eqref{eq:cont_dep} from \Cref{thm:wellposedness}, $\S \colon \Um \to \mathcal{Y}$ is Lipschitz continuous.
With this new notation, we define the reduced cost functional as
\begin{equation*}
    \J(u) \coloneqq \JJ(\S_1(u), \S_2(u), u),
\end{equation*}
where $\S_1$ and $\S_2$ select, respectively, the first and the second component of the solution operator $\S$.
Then, the optimal control problem can be stated as
\begin{equation}\label{eq:optimal_control_problem}
        \min_{u\in \Uad} \J(u),
\end{equation}
which means that we search a minimiser for the functional $\JJ$ subject to the PDE system \eqref{eq:system}--\eqref{eq:initial_datum} and constrained to $\Uad$. 

\begin{thm}\label{thm:existence_opt_control}
    The optimal control problem \eqref{eq:optimal_control_problem} has at least one minimiser $u^* \in \Uad$. 
\end{thm}

\begin{proof}
    The proof is quite standard, and it relies on the direct method of the Calculus of Variations. First of all, we notice that, since $\J$ is proper and nonnegative, $\inf_{u \in \Uad}\J(u)$ is finite and nonnegative. Let $\{u_n\}_{n=1}^{+\infty} \subseteq \Uad$ be a minimizing sequence, i.e.,
    \begin{equation*}
        \inf_{u \in \Uad}\J(u) = \lim_{n \to +\infty} \J(u_n),
    \end{equation*}
    with corresponding states denoted as $(\theta_n, \phi_n, \mu_n, \sigma_n) \coloneqq \S(u_n)$. Since $\Uad$ is bounded in $L^{\infty}(Q)$, there exists a $u^* \in L^{\infty}(Q)$ such that
    \begin{alignat}{2}
        u_n \to u^* &\quad \text{weakly-}\ast &&\quad \text{in } L^{\infty}(Q)\label{eq:w_conv_u_optcontrol}
    \end{alignat}
    along a non-relabelled subsequence.
    Moreover, since $\Uad$ is convex and closed in $L^{\infty}(Q)$, it is weakly-$\ast$ closed. Thus, the limit $u^*$ belongs to the admissible set $\Uad$. Regarding the states, first we recall the estimate \eqref{eq:solution_estimate}, the compactness of the embeddings $W\subset V\subset H$ and $W\in C^0(\overline\Omega) $, the
    Banach--Alaouglu (see, e.g., \cite{brezis2011}) and Aubin--Lions theorems (see \cite[][Section 8, Corollary 4]{Simon_86}). Then, by these tools, it turns out that there exists a quadruplet $(\theta^*, \phi^*, \mu^*, \sigma^*) \in \mathcal{X}$ such that
    \begin{alignat}{2}
         \theta_n \to \theta^* &\quad \text{weakly-}\ast &&\quad \text{in } H^1(0,T;\Hs) \cap L^\infty(0,T;\Vs)  \cap L^2(0,T;\Ws) \cap L^{\infty}(Q),\label{eq:w_conv_theta_optcontrol}\\
         &\quad \text{strongly} &&\quad \text{in } C^0([0,T];\Hs) \cap L^2(0,T;\Vs),\label{eq:s_conv_theta_optcontrol}\\
         \phi_n \to \phi^* &\quad \text{weakly-}\ast &&\quad \text{in } W^{1,\infty}(0,T;\Vs') \cap  H^1(0,T;\Vs) \cap  L^{\infty}(0,T;\Ws),\label{eq:w_conv_phi_optcontrol}\\
         &\quad \text{strongly} &&\quad \text{in } C^0([0,T];\Vs)\cap C^0(\overline{Q}),\label{eq:s_conv_phi_optcontrol}\\
         \mu_n \to \mu^* &\quad \text{weakly-}\ast &&\quad \text{in } L^{\infty}(0,T;\Vs) \cap  L^2(0,T;\Ws),\label{eq:w_conv_mu_optcontrol}\\
         \sigma_n \to \sigma^* &\quad \text{weakly-}\ast &&\quad \text{in } H^1(0,T;\Hs) \cap  L^\infty(0,T;\Vs) \cap L^2(0,T;\Ws),\label{eq:w_conv_sigma_optcontrol}
         \\
         &\quad \text{strongly} &&\quad \text{in } C^0([0,T];\Hs) \cap L^2(0,T;\Vs),\label{eq:s_conv_sigma_optcontrol}
    \end{alignat}
    along a further subsequence that, again, we do not relabel. We claim that $(\theta^*, \phi^*, \mu^*, \sigma^*)$ coincides with $\S(u^*)$. To prove this, it is enough to write the system satisfied by each element of the sequence, and pass to the limit as $n \to +\infty$, exploiting 
    convergences \eqref{eq:w_conv_u_optcontrol}--\eqref{eq:s_conv_sigma_optcontrol}. In particular, the uniform convergence of $\phi_n$ implied by \eqref{eq:s_conv_phi_optcontrol} is crucial to pass to the limit in the nonlinearity $f(\phi_n)$. The passage to the limit in the other nonlinear terms can be easily handled, for instance $ \lambdad \sigma_n \k(\theta_n )$ converges to  $\lambdad \sigma \k(\theta)$ strongly in $C^0([0,T];L^1(\Omega))$  due to \eqref{eq:s_conv_sigma_optcontrol}, \eqref{eq:s_conv_theta_optcontrol} and the Lipschitz continuity of $\k$ ensured by \ref{hyp:nonlinearities}.
    Consequently, $(\theta^*, \phi^*, \mu^*, \sigma^*)$ is the strong solution of the system \eqref{eq:system}--\eqref{eq:initial_datum} associated with the admissible control $u^*$. Moreover, by the weak 
    lower semicontinuity of the $L^2$ norm and the weak convergence \eqref{eq:w_conv_u_optcontrol} combined with the continuity of the $L^2$ norm and the strong convergences \eqref{eq:s_conv_theta_optcontrol}, \eqref{eq:s_conv_phi_optcontrol}, we have
    \begin{equation*}
        \J(u^*) \leq \lim_{n \to +\infty} \J(u_n) = \inf_{u \in \Uad} \J(u).
    \end{equation*}
    Therefore, $u^*$ is a minimiser of $\J$ over $\Uad$, and the proof is complete.
\end{proof}

\noindent Our next goal is to characterise the optimal controls by means of first-order necessary optimality conditions. The standard approach consists in proving that the solution operator is Fr\'echet differentiable between suitable spaces, which in turn easily implies the Fréchet differentiability of the reduced cost functional $\J$.
As a first step in this direction, we introduce the so-called linearised system, which will play a crucial role in the explicit characterisation of the derivative of $\S$. Its formulation and analysis are the subject of the following section.

\section{The linearised system}
\label{sec:lin}

Let $u \in \Um$ be fixed, and let $(\theta, \phi, \mu, \sigma)$ denote the associated state. For a given (sufficiently small) perturbation $h \in L^{\infty}(Q)$, we consider the perturbed control $u + h$ and its corresponding state. Owing to the continuous dependence result stated in \Cref{thm:wellposedness}, this state remains close to $(\theta, \phi, \mu, \sigma)$.
To analyse the difference between the two state systems, we formally linearise the nonlinear terms by means of first-order Taylor expansions. In this way, we find that the perturbation variables $(\zeta, \xi, \eta, \rho)$---which we shall henceforth refer to as the linearised variables---satisfy the following linear system (written now with formal equations and conditions):
\begin{subequations}\label{eq:lin_system}
    \begin{align}
        & \partial_t(\zeta +\ell\xi) - \Delta \zeta = h,\label{eq:lin_temp}\\
        & \partial_t\xi - \Delta \eta = (\lambdap \rho - \lambdae \zeta)\h(\phi) + (\lambdap \sigma - \lambdaa - \lambdae \theta)\h'(\phi)\xi,\label{eq:lin_ch1}\\
        & \tau \partial_t \xi - \Delta \xi + f'(\phi)\xi - \chi \rho -\Lambda \zeta = \eta, \label{eq:lin_ch2}\\
        & \begin{aligned}
            \partial_t \rho - \Delta (\rho - \chi \xi) = &- \lambdac \rho \h(\phi) - \lambdac \sigma \h'(\phi)\xi\\
            &- \lambdab \rho  - \lambdad \rho \k(\theta) - \lambdad \sigma \k'(\theta)\zeta,
        \end{aligned} \label{eq:lin_nutrient}
    \end{align}
\end{subequations}
in $Q$, complemented with the homogeneous Neumann boundary conditions
\begin{equation}\label{eq:boundary}
    \partial_{\nnu} \zeta = \partial_{\nnu} \xi = \partial_{\nnu} \eta = \partial_{\nnu} \rho = 0
\end{equation}
on $\partial \Omega \times (0,T)$, and the initial conditions
\begin{equation}\label{eq:lin_init_cond}
    \zeta(0) = \xi(0) = \rho(0) = 0
\end{equation}
in $\Omega$. 
We emphasise that although the formal derivation above requires $h$ to be sufficiently small in $L^{\infty}(Q)$, the linear system \eqref{eq:lin_system} is well-defined under the milder assumption $h \in L^2(Q)$.

\begin{prop}\label{prop:existence_lin}
    For every control $u \in \Um$ with associated state $(\theta,\phi,\mu,\sigma)$ and every $h \in L^2(Q)$, there exists a unique solution $(\zeta,\xi,\eta,\rho)$ to the system \eqref{eq:lin_system}--\eqref{eq:lin_init_cond}, meaning that the quadruple belongs to the spaces
    \begin{gather*}
        \zeta \in H^1(0,T;\Vs') \cap C^{0}([0,T];\Hs) \cap L^2(0,T;\Vs),\\
        \xi \in H^1(0,T;\Vs') \cap L^{\infty}(0,T;\Vs)\cap L^2(0,T;\Ws),\qquad \tau^{1/2}\xi \in H^1(0,T;\Hs),\\
        \eta \in L^2(0,T; \Vs),\\
        \rho \in H^1(0,T;\Hs) \cap C^{0}([0,T];\Vs) \cap L^2(0,T;\Ws),
    \end{gather*}
     the equations \eqref{eq:lin_temp}--\eqref{eq:lin_nutrient} and the boundary conditions \eqref{eq:boundary} are satisfied in the following variational sense
    {\allowdisplaybreaks
    \begin{subequations}\label{eq:lin_problem_integrals}
        \begin{align}
            & \duality{\partial_t (\zeta + \ell \xi)}{v}_{\Vs} + \into \nabla \zeta \cdot \nabla v \dx = \into h \, v \dx, \label{eq:lin_temp_weak} \\
            & \begin{aligned}
                &\duality{\partial_t \xi}{v}_{\Vs} + \into \nabla \eta \cdot \nabla v \dx\\
                & \quad = \into \left[(\lambdap \rho - \lambdae \zeta) \h(\phi) + (\lambdap \sigma - \lambdaa - \lambdae \theta) \h'(\phi) \xi \right] v \dx,
            \end{aligned} \label{eq:lin_ch1_weak} \\
            & \tau \duality{\partial_t \xi}{v}_{\Vs} + \into \nabla \xi \cdot \nabla v \dx + \into \left(f'(\phi) \xi - \chi \rho - \Lambda \zeta \right) v \dx = \into \eta \, v \dx, \label{eq:lin_ch2_weak} \\
            & \begin{aligned}
                &\duality{\partial_t \rho}{v}_{\Vs} + \into \nabla (\rho - \chi \xi) \cdot \nabla v \dx\\
                &\quad = \into \Big[ - \lambdac \rho \, \h(\phi) - \lambdac \sigma \h'(\phi) \xi - \lambdab \rho - \lambdad \rho \, \k(\theta) - \lambdad \sigma \k'(\theta) \zeta \Big] v \dx, 
            \end{aligned}\label{eq:lin_nutrient_weak}
        \end{align}
    \end{subequations}}%
    for all test functions $v \in \Vs$ and a.e. $t \in (0,T)$, and  the initial conditions \eqref{eq:lin_init_cond} are satisfied a.e. in $\Omega$.
\end{prop}

\begin{proof}
    The existence of solutions can be established by means of a Faedo--Galerkin space discretisation with a special basis. Since this procedure is standard, we omit the technical details. In what follows, we focus on (formally) deriving the a priori estimates at the continuous level; performing these estimates at the discrete level is permitted and gives enough compactness to pass to the limit, finding a solution of the system \eqref{eq:lin_system}--\eqref{eq:lin_init_cond}. From these very same estimates (more precisely, see the subsequent \eqref{eq:lin_est1}, \eqref{eq:lin_est3}) and the linearity of the system, uniqueness follows.\\
    
    \noindent \textbf{First estimate.} We integrate the equation \eqref{eq:lin_temp_weak} over the time interval $(0,t)$ for a $t \in (0,T)$, we multiply it by a (big) constant $R > 0$ fixed but yet to be chosen, and we take as a test function $v = \zeta$. We sum the resulting inequality with \eqref{eq:lin_ch1_weak} tested with $v = \xi$, \eqref{eq:lin_ch2_weak} tested with $v = (- \Delta \xi)$, and \eqref{eq:lin_nutrient_weak} tested with $v = \rho$. Some terms cancel out, and we obtain:
    \begin{align}\label{eq:lin_est1_1}
        &R \into |\zeta|^2 \dx + \frac{R}{2} \difft \into |\nabla (1 \ast_t \zeta)|^2 \dx + \frac{1}{2} \difft \into |\xi|^2 \dx + \frac{\tau}{2} \difft \into |\nabla \xi|^2 \dx \notag\\
        &\quad \quad + \into |-\Delta \xi|^2 \dx + \frac{1}{2} \difft \into |\rho|^2 \dx + \into |\nabla \rho|^2 \dx \notag\\
        &\quad = R \into (1 \ast_t h) \zeta \dx - R\ell \into \xi \zeta \dx \notag\\
        &\quad \quad + \into (\lambdap \rho - \lambdae \zeta) \h(\phi) \xi \dx + \into (\lambdap \sigma - \lambdaa - \lambdae \theta) \h'(\phi) |\xi|^2 \dx\\
        &\quad \quad - \into f'(\phi)\xi (-\Delta \xi) \dx + \chi \into \rho (-\Delta \xi) \dx + \Lambda \into \zeta (-\Delta \xi) \dx \notag\\
        &\quad \quad  + \chi \into (-\Delta \xi) \rho \dx -\lambdac \into \sigma \h'(\phi) \xi \rho \dx - \lambdad \into \sigma \k'(\theta)\zeta \rho \notag\\
        &\quad \quad - \into \left(\lambdac \h(\phi) + \lambdab + \lambdad \k(\theta)\right)|\rho|^2 \dx. \notag
    \end{align}
    We aim to estimate the right-hand side of \eqref{eq:lin_est1_1} from above. First of all, we notice that the last addend is nonpositive, because by hypotheses \ref{hyp:constants} and \ref{hyp:nonlinearities} the constants $\lambdac$, $\lambdab$, $\lambdad$ and the assigned functions $\h$, $\k$ are nonnegative. Then, we apply the H\"older and the Young inequalities, sometimes with a (small) constant $\epsilon > 0$ yet to be defined. We also make use of the continuous embedding $\Ws \subseteq L^{\infty}(\Omega)$, which holds in dimension $d=2,3$. 
   Moreover, we take advantage of the estimate in \eqref{eq:solution_estimate}  and, in particular, of the $L^\infty(Q)$-bound for $\theta$.
    Rearranging some terms, we have:
    \begin{equation*}
        \begin{split}
            & \frac{1}{2} \difft \left[ R \into |\nabla (1 \ast_t \zeta)|^2 \dx + \into |\xi|^2 \dx + \tau \into |\nabla \xi|^2 \dx + \into |\rho|^2 \dx \right] \\
            &\quad \quad + R \into |\zeta|^2 \dx + \into |-\Delta \xi|^2 \dx + \into |\nabla \rho|^2 \dx\\
            &\quad \leq \left(\frac{R}{2} + C_{\epsilon}\right) \into |\zeta|^2 \dx + \epsilon \into |-\Delta \xi|^2 \dx\\
            &\quad \quad + C_{\epsilon,R} \bigg[\into |1 \ast_t h|^2 \dx 
            + (1 + \norm{\sigma}_{\Ws}^2)\into |\rho|^2 \dx + (1 + \norm{\sigma}_{\Ws}^2)\into |\xi|^2 \dx \bigg]. 
        \end{split}
    \end{equation*}
    Next, we chose $\epsilon$ small enough (e.g., $\epsilon = 1/2$) and $R$ big enough (so that the now fixed $C_{\epsilon}$ is strictly smaller that $R/2$). This way, the first two addends on the right-hand side of the inequality can be absorbed into the corresponding terms on the left-hand side. Renaming the constants, integrating in time over the time interval $(0,t)$, and exploiting the vanishing initial conditions \eqref{eq:lin_init_cond} leads to
    \begin{equation*}
        \begin{split}
             &   \into |\nabla (1 \ast_t \zeta)|^2 \dx + \into |\xi|^2 \dx + \tau \into |\nabla \xi|^2 \dx + \into |\rho|^2 \dx \\
            &\quad \quad + R \intto |\zeta|^2 \dx \ds + \intto |-\Delta \xi|^2 \dx \ds + \intto |\nabla \rho|^2 \dx \ds \\
            &\quad \leq C \int_0^t \bigg(\into |1 \ast_t h|^2 \dx 
            + (1 + \norm{\sigma}_{\Ws}^2)\into |\rho|^2 \dx + (1 + \norm{\sigma}_{\Ws}^2)\into |\xi|^2 \dx \bigg) \ds. 
        \end{split}
    \end{equation*}
    We apply the Gronwall inequality, obtaining:
    \begin{equation}\label{eq:lin_est1}
        \begin{split}
            &\norm{\zeta}_{L^2(\Hs)} + \norm{1 \ast_t \zeta}_{H^1(\Hs) \cap L^{\infty}(\Vs)} + \norm{\xi}_{L^{\infty}(\Hs) \cap L^2(\Ws)} + \norm{\tau^{1/2} \xi}_{L^{\infty}(\Vs)}\\
            & \quad+ \norm{\rho}_{L^{\infty}(\Hs) \cap L^2(\Vs)} \leq C \norm{1 \ast_t h}_{L^2(\Hs)} \leq C \norm{h}_{L^2(\Hs)},
        \end{split}
    \end{equation}
    where we have also exploited the uniform bound of $\sigma$ in the norm $L^2(0,T;\Ws)$ given by the inequality \eqref{eq:solution_estimate} in \Cref{thm:wellposedness}.\\

    \noindent \textbf{Second estimates.} We test \eqref{eq:lin_ch1_weak} with $v = 1$ and take the absolute value of both sides of the equality and use the bound on $\theta$ in $L^\infty(Q)$, obtaining 
    \begin{equation*}
        \begin{split}
            |\overline{\partial_t \xi}| & =\left| \into \left[(\lambdap \rho - \lambdae \zeta) \h(\phi) + (\lambdap \sigma - \lambdaa - \lambdae \theta) \h'(\phi) \xi \right] \dx \right|\\
            & \leq  C \into \left[|\rho| + |\zeta| + (1 + |\sigma|+|\theta|)|\xi|\right] \dx\\
            & \leq C \left[\norm{\rho}_{\Hs} + \norm{\zeta}_{\Hs} + (1 + \norm{\sigma}_{\Hs})\norm{\xi}_{\Hs} \right].
        \end{split}
    \end{equation*}
    Taking the $L^2(0,T)$-norm of both sides of the inequality, exploiting the estimates \eqref{eq:solution_estimate} and \eqref{eq:lin_est1}, we have
    \begin{equation}\label{eq:lin_est2_1}
        \norm{\overline{\partial_t \xi}}_{L^2(0,T)} \leq C \norm{h}_{L^2(\Hs)}.
    \end{equation}
    Similarly, we test the equation \eqref{eq:lin_ch2_weak} with $v = 1$, employ the estimates \eqref{eq:solution_estimate}, \eqref{eq:lin_est1}, \eqref{eq:lin_est2_1}, and derive
    \begin{equation}\label{eq:lin_est2_2}
        \norm{\overline{\eta}}_{L^2(0,T)} \leq C \norm{h}_{L^2(\Hs)}.
    \end{equation}

    \bigskip

    \noindent \textbf{Third estimate.} We sum the equation \eqref{eq:lin_temp_weak} tested with $v = \Lambda \zeta/\ell$, \eqref{eq:lin_ch1_weak} tested with $v = \eta$, \eqref{eq:lin_ch2_weak} tested with $v = \partial_t \xi$, and \eqref{eq:lin_nutrient_weak} tested with $v = \partial_t \rho$. After some cancellations, we find out that
\begin{equation}
    \begin{split}
        &\frac{1}{2} \difft \biggl(\frac{\Lambda}{\ell}\into |\zeta|^2 \dx \biggr) + \frac{\Lambda}{\ell}\into |\nabla \zeta|^2 \dx + \into |\nabla \eta|^2 \dx + \tau \into |\partial_t \xi|^2 \dx  \dx\\
        & \quad \quad + \frac{1}{2} \difft \into |\nabla \xi|^2
        + \into |\partial_t \rho|^2 \dx + \frac{1}{2} \difft \into |\nabla \rho|^2 \dx + \frac{\lambdab}{2} \difft \into |\rho|^2 \dx\\
        &\quad = \frac{\Lambda}{\ell}\into h \zeta \dx + \into (\lambdap\rho - \lambdae\zeta)\h(\phi) \eta \dx + \into (\lambdap\sigma - \lambdaa - \lambdae\theta)\h'(\phi) \xi \eta \dx \\
        &\quad \quad - \into f'(\phi) \xi \partial_t \xi \dx + \chi \into \rho \partial_t \xi \dx - \into \lambdac \h(\phi) \rho \partial_t \rho \dx\\
        &\quad \quad - \into \lambdac \sigma \h'(\phi)\xi\partial_t\rho \dx - \into \lambdad \k(\theta) \rho \partial_t \rho \dx - \into \lambdad \k'(\theta)\sigma\zeta \partial_t\rho \dx\\
        &\quad \quad + \chi \into (-\Delta \xi) \partial_t \rho \dx.
    \end{split}
\end{equation}
We integrate in time over $(0,t)$, exploit the initial conditions \eqref{eq:lin_init_cond}, and estimate the right-hand side. In particular, we exploit the fact that $\h$, $\k$ are Lipschitz continuous given by hypothesis \ref{hyp:given_function}, $f'$ being continuous by hypothesis \ref{hyp:beta_pi}, and $\theta$, $\phi$ being uniformly bounded in $L^{\infty}(Q)$ by the estimate \eqref{eq:solution_estimate}. By the H\"older and the Young inequalities, rearranging the terms on the left-hand side, we have
{\allowdisplaybreaks
\begin{equation}\label{eq:lin_est3_0}
    \begin{split}
        &\frac{1}{2} \left[\frac{\Lambda}{\ell}\into |\zeta|^2 \dx + \into |\nabla \xi|^2 \dx + \into |\nabla \rho|^2 \dx  + \lambdab \into |\rho|^2 \dx \right]  \\
        & \quad \quad{} +  \frac{\Lambda}{\ell}\intto |\nabla \zeta|^2 \dx \ds + \intto |\nabla \eta|^2 \dx \ds
        \\
        & \quad \quad{} 
        + \tau \intto |\partial_t \xi|^2 \dx \ds + \intto |\partial_t \rho|^2 \dx \ds\\
        &\quad \leq \intto |h|^2 \dx \ds + C \intto |\zeta|^2 \dx \ds + I_1 + I_2 + I_3 + I_4\\
        &\quad \quad + \epsilon \intto |\partial_t \rho|^2 \dx \ds + C_{\epsilon} \bigg[ \intto |\rho|^2 \dx \ds+ \intt \norm{\sigma}_{\Ws}^2 \norm{\xi}_{\Hs}^2\ds\\
        &\quad \quad + \intt \norm{\sigma}_{\Ws}^2 \norm{\zeta}_{\Hs}^2\ds + \intto |-\Delta \xi|^2 \dx \ds \bigg]\\
        &\quad \leq \epsilon \intto |\partial_t \rho|^2 \dx \ds + I_1 + I_2 + I_3 + I_4\\
        & \quad \quad + C_{\epsilon} \left[ \norm{h}_{L^2(\Hs)}^2 + \intt \norm{\sigma}_{\Ws}^2  \norm{\zeta}_{\Hs}^2\ds \right],
    \end{split}
\end{equation}%
}%
for a small constant $\epsilon>0$ fixed but yet to be determined, where we used the uniform estimates \eqref{eq:solution_estimate} and \eqref{eq:lin_est1}. In the following, we analyse the terms $I_1, \dots, I_4$ singularly. 
Regarding $I_1$, we apply the Young inequality, where $\epsilon$ is as before, and the Poincar\'e--Wirtinger inequality to the $\eta$ term, subtracting $\overline{\eta}$. Note that $C_P$ is a fixed constant
related to this inequality. Finally, we employ the H\"older inequality and exploit~\eqref{eq:lin_est1} and \eqref{eq:lin_est2_2}, obtaining: 
\begin{equation}\label{eq:lin_est3_1}
    \begin{split}
        I_1 & = \intto (\lambdap\rho - \lambdae\zeta) \h(\phi) \eta \dx \ds\\
        &\leq \frac{\epsilon}{C_P} \intto |\eta|^2 \dx \ds + C_{\epsilon} \intto \left(|\rho|^2 + |\zeta|^2\right) \dx \ds\\
        &\leq \epsilon \intto |\nabla \eta|^2 \dx \ds + C_{\epsilon} \intto \left(|\rho|^2 + |\zeta|^2 \right) \dx \ds + C\intt |\overline{\eta}|^2 \ds \\
        &\leq \epsilon \intto |\nabla \eta|^2 \dx \ds + C_{\epsilon} \norm{h}_{L^2(\Hs)}^2.
    \end{split}
\end{equation}
We turn our attention to $I_2$ and proceed similarly:
\begin{equation}\label{eq:lin_est3_2}
    \begin{split}
        I_2 &= \intto (\lambdap\sigma - \lambdaa - \lambdae\theta) \h'(\phi) \xi \eta \dx \ds \leq C \intto ( |\sigma| + 1 )|\xi| |\eta| \dx \ds \\
        &\leq \epsilon \intto |\eta|^2 \dx \ds + C_{\epsilon} \intt (\norm{\sigma}_{\Ws}^2 + 1) \norm{\xi}_{\Hs}^2 \ds\\
        & \leq \epsilon \intto |\nabla \eta|^2 \dx \ds + C \intt |\overline{\eta}|^2 \ds + C_{\epsilon}\norm{\xi}_{L^\infty(\Hs)}^2\\
        &\leq \epsilon \intto |\nabla \eta|^2 \dx \ds + C_{\epsilon} \norm{h}_{L^2(\Hs)}^2,
    \end{split}
\end{equation}
having used the H\"older and the Poincaré inequalities, and the uniform estimates \eqref{eq:solution_estimate}, \eqref{eq:lin_est1}, \eqref{eq:lin_est2_2}. The term $I_3$ is first of all rewritten equivalently, and then integrated by parts, recalling the initial condition $\xi(0) = 0$ and \eqref{eq:solution_estimate} together with \eqref{hyp:hatf_reg} and \eqref{eq:lin_est1}:
\begin{equation}\label{eq:lin_est3_3}
    \begin{split}
        I_3 &= - \intto f'(\phi) \xi \partial_t \xi \dx \ds = - \frac{1}{2} \intto f'(\phi) \partial_t (|\xi|^2) \dx \ds\\
        &{}= - \frac{1}{2} \into \bigl(f'(\phi) |\xi|^2\bigr)\dx+ \frac{1}{2} \intto f''(\phi) \partial_t \phi |\xi|^2 \dx \ds\\
        &\leq C \left( \into |\xi|^2 \dx + \norm{\xi}_{L^{\infty}(\Hs)} \intt \norm{\partial_t \phi}_{\Hs} \norm{\xi}_{\Ws} \ds \right)\\
        & \leq C \left(\norm{\xi}_{L^{\infty}(\Hs)}^2 + \norm{\xi}_{L^{\infty}(\Hs)} \norm{\partial_t \phi}_{L^2(\Hs)} \norm{\xi}_{L^2(\Ws)} \right) \leq C  \norm{h}_{L^2(\Hs)}^2 .
    \end{split}
\end{equation}
The term $I_4$ is again dealt with by integrating by parts, exploiting the vanishing initial conditions:
\begin{equation}\label{eq:lin_est3_4}
    \begin{split}
         I_4 &= \chi \intto \rho \partial_t \xi \dx \ds = \chi \into (\rho \xi)(t)\dx - \chi\intto \partial_t \rho \xi \dx \ds\\
        &\leq \chi \norm{\rho}_{L^{\infty}(\Hs)} \norm{\xi}_{L^{\infty}(\Hs)} + \epsilon \intto |\partial_t \rho|^2 \dx \ds + C_\epsilon \intto \xi^2 \dx \ds\\
        & \leq \epsilon \intto |\partial_t \rho|^2 \dx \ds + C_{\epsilon} \norm{h}_{L^2(\Hs)}^2,
    \end{split}
\end{equation}
by the H\"older inequality and the estimate \eqref{eq:lin_est1}. Finally, we collect the inequalities \eqref{eq:lin_est3_0}--\eqref{eq:lin_est3_4}, choosing $\epsilon$ small enough so that the related terms on the right-hand side can be absorbed in the corresponding addends on the left-hand side. This leads to the inequality
\begin{equation*}
    \begin{split}
         &\frac{1}{2} \left[ \frac{\Lambda}{\ell}\into |\zeta|^2 \dx + \into |\nabla \xi|^2 \dx + \into |\nabla \rho|^2 \dx  + \lambdab \into |\rho|^2 \dx \right] +  \frac{\Lambda}{\ell}\intto |\nabla \zeta|^2 \dx \ds \\
        & \quad \quad + \frac12\intto |\nabla \eta|^2 \dx \ds+ \tau \intto |\partial_t \xi|^2 \dx \ds + \frac12\intto |\partial_t \rho|^2 \dx \ds\\
        &\quad \leq C \left(\norm{h}_{L^2(\Hs)}^2 + \intt \norm{\sigma}_{\Ws}^2 \norm{\zeta}_{\Hs}^2\ds \right).
    \end{split}
\end{equation*}
Thus, as the function $s\mapsto \norm{\sigma(s)}_{\Ws}^2$ is in $L^1(0,T) $ by \eqref{eq:solution_estimate}, we can apply the Gronwall inequality and obtain:
\begin{equation}\label{eq:lin_est3}
    \begin{split}
            &\norm{\zeta}_{L^\infty(H) \cap L^2(V)} + \norm{\xi}_{L^\infty(V)} + \norm{\tau^{1/2} \partial_t \xi}_{L^2(H)} \\
            &\quad+ \norm{\eta}_{L^2(V)} + \norm{\rho}_{H^1(H)\cap L^\infty(V)} \leq C \norm{h}_{L^2(H)}.
    \end{split}
\end{equation}
Moreover, by a comparison argument applied first to \eqref{eq:lin_ch1_weak} 
and then to \eqref{eq:lin_temp_weak}, we have:
\begin{equation}\label{eq:lin_est3_add}
    \norm{\partial_t \xi}_{L^2(V')} \leq C \norm{h}_{L^2(H)}, \quad \norm{\partial_t \zeta}_{L^2(V')}  \leq C \norm{h}_{L^2(H)}.
\end{equation}
\end{proof}

\section{The Fréchet differentiability of the solution operator}

\label{sec:diff}

In this section, we show that the solution operator $\S : \Um \to \mathcal{Y}$ is not only Lipschitz continuous, but in fact Fréchet differentiable.

\begin{thm}\label{thm:diff_solution_op}
The solution operator $\S : \Um \to \mathcal{Y}$ is Fréchet differentiable. Moreover, for any $u \in \Um$, the Fréchet derivative of $\S$ at $u$ is characterised by
\begin{equation*}
        D\S(u)[h] = (\zetah,\xih,\etah,\rhoh)
    \end{equation*}
for all $h \in \U$, where $(\zetah,\xih,\etah,\rhoh)$ denotes the solution of the linearised system \eqref{eq:lin_system}--\eqref{eq:lin_init_cond}.
\end{thm}

\begin{proof}
Let $u \in \Uad$ be fixed but arbitrary and denote $\S(u) = (\theta, \phi, \mu, \sigma)$. We aim to show that
\begin{equation}\label{eq:limit_diff}
\lim_{\norm{h}_{\U} \to 0}
\frac{\norm{\S(u+h) - \S(u) - (\zetah, \xih, \etah, \rhoh)}_{\mathcal{Y}}}{\norm{h}_{\U}} = 0,
\end{equation}
which proves the asserted Fréchet differentiability.
To this end, we set $\S(u+h) = (\thetah, \phih, \muh, \sigmah)$ and introduce the remainder terms
\begin{equation*}
    \Thetah \coloneqq \thetah - \theta - \zetah, 
    \quad \Phih \coloneqq \phih - \phi - \xih, 
    \quad \Upsilonh \coloneqq \muh - \mu - \etah,
    \quad \Pih \coloneqq \sigmah - \sigma - \rhoh.
\end{equation*}
Since $u \in \Um$ and $\Um$ is open, there exists a constant $C_{u} > 0$ such that $u+h \in \Um$ for all $h \in \U$ satisfying $\norm{h}_{\U} \leq C_u$. As we are interested in the limit as $\norm{h}_{\U} \to 0$, we may restrict ourselves, without loss of generality, to perturbations $h$ of sufficiently small norm.
We claim that
\begin{equation}\label{eq:diff_thesis}
\norm{(\Thetah, \Phih,\Upsilonh, \Pih)}_{\mathcal{Y}} \le C_M \norm{h}_{\U}^{2},
\end{equation}
for a suitable constant $C_M > 0$. Clearly, estimate \eqref{eq:diff_thesis} immediately implies \eqref{eq:limit_diff}.
In order to prove \eqref{eq:diff_thesis}, we analyse the PDE system satisfied by $(\Thetah, \Phih, \Upsilonh, \Pih)$, which is obtained by subtracting the state system written for $u$ and its linearization (cf. \Cref{prop:existence_lin}) from the state system corresponding to $u+h$. By writing the equations formally, it reads as follows:
{\allowdisplaybreaks
    \begin{subequations}\label{eq:diff_problem}
        \begin{align}
            & \partial_t (\Thetah + \ell \Phih) - \Delta \Thetah = 0, \label{eq:diff_temp} \\
            &\begin{aligned}
                \partial_t \Phih - \Delta \Upsilonh = & \, (\lambdap \sigmah - \lambdaa - \lambdae \thetah) \h(\phih) - (\lambdap \sigma - \lambdaa - \lambdae \theta) \h(\phi)\\
                & - \left[(\lambdap \rhoh - \lambdae \zetah)\h(\phi) + (\lambdap \sigma - \lambdaa - \lambdae \theta) \h'(\phi) \xih \right],
            \end{aligned} \label{eq:diff_ch1} \\
            &\tau \partial_t \Phih - \Delta \Phih + f(\phih) - f(\phi)- f'(\phi)\xih - \chi \Phih - \Lambda \Thetah =  \Upsilonh, \label{eq:diff_ch2} \\
            &\begin{aligned}
                \partial_t \Pih - \Delta(\Pih - \chi \Phih) = & \, \lambdac \big(-\sigmah \h(\phih) + \sigma \h(\phi) + \rhoh \h(\phi) + \sigma \h'(\phi) \xih \big) - \lambdab \Pih \\
                &+ \lambdad \big(-\sigmah \k(\thetah) + \sigma \k(\theta) + \rhoh \k(\theta) + \sigma \k'(\theta) \zetah \big), 
            \end{aligned}\label{eq:diff_nutrient}
        \end{align}
    \end{subequations}
    }%
    coupled with vanishing initial and Neumann boundary conditions.\\
    
    \noindent \textbf{Rewriting the nonlinearities.} We begin by rewriting the nonlinear contributions in a more convenient form. Consider first the right-hand side of \eqref{eq:diff_ch1}. By adding and subtracting suitable terms, we obtain
    \begin{equation*}
        \begin{split}
            &(\lambdap \sigmah - \lambdaa - \lambdae \thetah) \h(\phih) - (\lambdap \sigma - \lambdaa - \lambdae \theta) \h(\phi)\\
            & \quad \quad- \left[(\lambdap \rhoh - \lambdae \zetah)\h(\phi) + (\lambdap \sigma - \lambdaa - \lambdae \theta) \h'(\phi) \xih \right]\\
            & \quad = (\lambdap \sigma - \lambdaa - \lambdae \theta) \left[\h(\phih) - \h(\phi) - \h'(\phi)\xih\right]\\
            & \quad \quad + \left[\lambdap(\sigmah - \sigma) - \lambdae(\thetah - \theta)\right]\left(\h(\phih)-\h(\phi)\right) + (\lambdap \Pih - \lambdae \Thetah) \h(\phi).
        \end{split}
    \end{equation*}
    Now we take into account Taylor's expansion with integral remainder, which states that for every $g \in C^2(\RR)$ the following equality holds
    \begin{equation*}
        g(s_2) = g(s_1) + g'(s_1)(s_2 - s_1) + (s_2 - s_1)^2 \int_0^1 g''(z s_2 + (1-z)s_1)(1-z) \, \text{d}z 
    \end{equation*}
    for every $s_1, s_2 \in \RR$. Applying this formula with $g = \h$, $s_2 = \phih$, and $s_1 = \phi$, we deduce
    \begin{equation*}
        \h(\phih) - \h(\phi) - \h'(\phi) \xih = \h'(\phi) \Phih + R_{\h} (\phih - \phi)^2,
    \end{equation*}
    where the remainder term is given by
    \begin{equation*}
        R_{\h} \coloneqq \int_0^1 \h''\big(z\phih + (1-z)\phi\big)(1-z) \, \text{d}z.
    \end{equation*}
    Notice that, since $\h \in C^2(\RR)$ and both $\phih$ and $\phi$ are uniformly bounded in view of \eqref{eq:solution_estimate}, $R_{\h}$ is uniformly bounded by a constant depending only on $M$ (cf.~\eqref{betti1}). An analogous argument applies to the nonlinear potential term in \eqref{eq:diff_ch2}. Using Taylor's formula with integral remainder again, we obtain
    \begin{equation*}
        \begin{split}
            f(\phih) - f(\phi)- f'(\phi)\xih = f'(\phi)\Phih + R_{f} (\phih - \phi)^2,
        \end{split}
    \end{equation*}
    where
    \begin{equation*}
        R_{f} \coloneqq \int_0^1 f''\big(z\phih + (1-z)\phi\big) (1-z) \text{d}z.
    \end{equation*}
    Also in this case, $R_f$ is uniformly bounded.
    Proceeding similarly, we deal with the right-hand side of equation \eqref{eq:diff_nutrient}, obtaining:
    \begin{equation*}
        \begin{split}
            &\lambdac \big(-\sigmah \h(\phih) + \sigma \h(\phi) + \rhoh \h(\phi) + \sigma \h'(\phi) \xih \big) - \lambdab \Pih \\
            & \quad \quad + \lambdad \big(-\sigmah \k(\thetah) + \sigma \k(\theta) + \rhoh \k(\theta) + \sigma \k'(\theta) \zetah \big)\\
            & \quad = \lambdac \Big\{-\sigma\left[ \h(\phih) - \h(\phi) - \h'(\phi)\xih\right] - (\sigmah - \sigma)  \big(\h(\phih) - \h(\phi)\big) - \Pih \h(\phi)\Big\}\\
            & \quad \quad - \lambdab \Pih + \lambdad \Big\{ - \sigma \left[\k(\thetah) - \k(\theta) - \k'(\theta)\zetah\right] - (\sigmah - \sigma)\big(\k(\thetah) - \k(\theta)\big) - \Pih \k(\theta)\Big\}\\
            & \quad = \lambdac \Big\{-\sigma\left[\h'(\phi)\Phih + R_{\h}(\phih - \phi)^2\right] - (\sigmah - \sigma)  \big(\h(\phih) - \h(\phi)\big) - \Pih \h(\phi)\Big\}\\
            & \quad \quad - \lambdab \Pih + \lambdad \Big\{ - \sigma \left[\k'(\theta)\Thetah + R_{\k}(\thetah - \theta)^2\right] - (\sigmah - \sigma)\big(\k(\thetah) - \k(\theta)\big) - \Pih \k(\theta)\Big\},
        \end{split}
    \end{equation*}
    where $R_{\k}$ is defined as before, i.e.,
    \begin{equation*}
        R_{\k} \coloneqq \int_0^1 \k''\big(z\thetah + (1-z)\theta\big)(1-z) \,\text{d}x,
    \end{equation*}
    and uniformly bounded. With the notation introduced, $(\Thetah, \Phih,\Upsilonh, \Pih)$ satisfies the following PDE system written in a variational formulation:
    {
    \allowdisplaybreaks
    \begin{subequations}\label{eq:diff_problem_integrals}
        \begin{align}
            & \duality{\partial_t (\Thetah + \ell \Phih)}{v}_{\Vs} + \into \nabla \Thetah \cdot \nabla v \dx = 0, \label{eq:diff_temp_weak} \\
            &\begin{aligned}
                &\duality{\partial_t \Phih}{v}_{\Vs} + \into \nabla \Upsilonh \cdot \nabla v \dx \\
                &\quad = \into (\lambdap \sigma - \lambdaa - \lambdae \theta) \left[\h'(\phi) \Phih + R_{\h} (\phih -\phi)^2\right] v \dx\\
                & \quad \quad + \into \left[\lambdap(\sigmah - \sigma) - \lambdae(\thetah - \theta)\right]\left(\h(\phih)-\h(\phi)\right)v \dx\\
                &\quad \quad + \into (\lambdap \Pih - \lambdae \Thetah) \h(\phi) v \dx,
                \end{aligned} \label{eq:diff_ch1_weak} \\
            &\begin{aligned}
                 &\tau \duality{\partial_t \Phih}{v}_{\Vs} + \!\into \! \nabla \Phih \cdot \nabla v \dx \\
                 &\quad + \!\into \!\! \left[ f'(\phi)\Phih + R_{f} (\phih - \phi)^2 - \chi \Phih - \Lambda \Thetah \right] v \dx = \into \Upsilonh \, v \dx, 
            \end{aligned}\label{eq:diff_ch2_weak} \\
            &\begin{aligned}
                &\duality{\partial_t \Pih}{v}_{\Vs} + \into \nabla (\Pih - \chi \Phih) \cdot \nabla v \dx\\
                &\quad = \into \lambdac \Big\{-\sigma\left[\h'(\phi)\Phih + R_{\h}(\phih - \phi)^2\right] - (\sigmah - \sigma)  \big(\h(\phih) - \h(\phi)\big)\Big\} v \dx\\
                & \quad \quad - \into \Pih \big(\lambdac\h(\phi) + \lambdab + \lambdad \k(\theta) \big)v \dx\\
                & \quad \quad + \into \lambdad \Big\{ - \sigma \left[\k'(\theta)\Thetah + R_{\k}(\thetah - \theta)^2\right] - (\sigmah - \sigma)\big(\k(\thetah) - \k(\theta)\big)\Big\} v \dx,
            \end{aligned}\label{eq:diff_nutrient_weak}
        \end{align}
    \end{subequations}
    }%
    for every $v \in \Vs$, a.e. in $(0,T)$. In order to prove the inequality \eqref{eq:diff_thesis}, we are going to perform some a priori estimates.\\

    \noindent \textbf{First estimates.} We integrate the equation \eqref{eq:diff_temp_weak} in time, multiply the resulting identity by a (large) constant $R>0$ to be chosen later, and choose $v = \Thetah$ as a test function. Then, we integrate once more in time over the interval $(0,t)$. Exploiting the homogeneous initial condition for $\Thetah$, we obtain: 
    \begin{equation}\label{eq:diff_est1_1}
        \begin{split}
            &R \intto |\Thetah|^2 \dx \ds + \frac{R}{2} \into |\nabla (1 \ast_t \Thetah)|^2 \dx\\
            &\quad = R \ell \intto \Phih \Thetah \dx \ds \leq \frac R 4 \intto |\Thetah|^2 \dx \ds + R\ell^2 \intto |\Phih|^2 \dx \ds,
        \end{split}
    \end{equation}
    where we have bounded the right-hand side by the Young inequality.
    
    Next, we sum \eqref{eq:diff_ch1_weak} tested with $v = \Phih$ and \eqref{eq:diff_ch2_weak} tested with $v = - \Delta \Phih$: the use of the latter function is not allowed but \eqref{eq:diff_ch2_weak} can be equivalently rewritten as the equation \[ \tau \partial_t \Phih - \Delta \Phih +  f'(\phi)\Phih + R_{f} (\phih - \phi)^2 - \chi \Phih - \Lambda \Thetah  = \Upsilonh \quad \hbox{a.e. in $Q$},\]  and here we can test by $( - \Delta \Phih)$. Integrating by parts, a cancellation of the terms involving $\Upsilonh$ occurs. The integration in time over $(0,t)$ and the use of the null initial conditions entail that
    \begin{equation}\label{eq:diff_est1_2}
        \begin{split}
            &\frac{1}{2} \into |\Phih|^2 \dx + \frac{\tau}{2} \into |\nabla \Phih|^2 \dx + \intto |-\Delta \Phih|^2 \dx \ds\\
            & \quad = \intto (\lambdap \sigma - \lambdaa - \lambdae \theta) \left[\h'(\phi) \Phih + R_{\h} (\phih -\phi)^2\right] \Phih \dx \ds\\
            & \quad \quad + \intto \left[\lambdap(\sigmah - \sigma) - \lambdae(\thetah - \theta)\right]\left(\h(\phih)-\h(\phi)\right) \Phih \dx \ds\\
            &\quad \quad + \intto (\lambdap \Pih - \lambdae \Thetah) \h(\phi) \Phih \dx \ds\\
            & \quad \quad - \intto \left[ f'(\phi)\Phih + R_{f} (\phih - \phi)^2 - \chi \Phih - \Lambda \Thetah \right] (-\Delta \Phih) \dx \ds\\
            & \quad \eqqcolon I_1 + I_2 + I_3 + I_4.
        \end{split}
    \end{equation}
    Our next goal is to estimate the terms $I_1, \dots, I_4$ separately. 
    Concerning the term $I_1$, we first exploit that $\theta$ is uniformly bounded in $L^{\infty}(Q)$ by virtue of \eqref{eq:solution_estimate} in \Cref{thm:wellposedness}, that $\h$ is Lipschitz continuous in view of hypothesis~\ref{hyp:nonlinearities}, and that the remainder $R_{\h}$ is uniformly bounded, as observed above. We then apply H\"older’s inequality together with the continuous embeddings $\Vs \hookrightarrow L^4(\Omega)$ and $\Ws \hookrightarrow L^{\infty}(\Omega)$. Recalling again from \eqref{eq:solution_estimate} that $\sigma$ is uniformly bounded in $L^{\infty}(0,T;\Vs)$, we deduce:
    \begin{equation*}
        \begin{split}
            |I_1| &\leq C \intto (|\sigma| + 1 ) \big(|\Phih| + |\phih - \phi|^2 \big)|\Phih| \dx \ds\\
            &\leq C \intt \big(\norm{\sigma}_{L^4(\Omega)} + 1\big)\big( \norm{\Phih}_{\Hs} + \norm{\phih - \phi}_{\Hs}\norm{\phih - \phi}_{L^{\infty}(\Omega)}\big)\norm{\Phih}_{L^4(\Omega)} \ds\\
            & \leq C\left\{ \intt \norm{\Phih}_{\Hs}^2 \ds + \norm{\phih - \phi}_{L^{\infty}(\Hs)}^2\intt \norm{\phih - \phi}_{\Ws}^2 \ds + \intt \norm{\Phih}_{\Vs}^2 \ds\right\}.
        \end{split}
    \end{equation*}
    In order to bind the last addend, we invoke the inequality \eqref{pier6} and use the continuous dependence estimate \eqref{eq:cont_dep} from \Cref{thm:wellposedness}, obtaining
    \begin{equation}
        \begin{split}
            |I_1| &\leq \frac18 \intt \norm{-\Delta \Phih}_{\Hs}^2 \ds + C \intt \norm{\Phih}_{\Hs}^2 \ds + C \norm{h}_{L^2(\Hs)}^4
        \end{split}
    \end{equation}
    To estimate $I_2$, we again exploit the Lipschitz continuity of $\h$, H\"older’s and Young’s inequalities (with a small parameter $\epsilon>0$ to be fixed later), the interpolation inequality~\eqref{pier7}, and the continuous dependence estimate \eqref{eq:cont_dep}. We deduce that
    \begin{equation}
        \begin{split}
            |I_2| &\leq C \intto \big(|\sigmah - \sigma| + |\thetah - \theta|\big)|\phih - \phi||\Phih| \dx \ds\\
            & \leq C \intt \big( \norm{\sigmah - \sigma}_{\Hs} + \norm{\thetah - \theta}_{\Hs}\big) \norm{\phih - \phi}_{\Hs} \norm{\Phih}_{L^{\infty}(\Omega)} \ds\\
            & \leq \intt \norm{\Phih}_{L^{\infty}(\Omega)}^2 \ds + C \norm{\phih - \phi}_{L^{\infty}(\Hs)}^2 \intt \big( \norm{\sigmah - \sigma}_{\Hs}^2 + \norm{\thetah - \theta}_{\Hs}^2\big) \ds    \\
            & \leq \frac18 \intt \norm{-\Delta \Phih}_{\Hs}^2 \ds + C \intt \norm{\Phih}_{\Hs}^2 \ds + C \norm{h}_{L^2(\Hs)}^4 .\\
        \end{split}
    \end{equation}
    We now turn to $I_3$. In this case, it suffices to exploit the boundedness of $\h$, together with Hölder’s and Young’s inequalities. In view of \eqref{eq:diff_est1_1}, which later will be added to \eqref{eq:diff_est1_2}, we infer that 
    \begin{equation}\label{pier11}
        \begin{split}
            |I_3| &\leq C \intto \big(|\Pih| + |\Thetah|\big)|\Phih| \dx \ds\\ &\leq \frac R 4  \intt \norm{\Thetah}_{\Hs}^2 \ds + C_{R} \intt \big(\norm{\Pih}_{\Hs}^2 + \norm{\Phih}_{\Hs}^2 \big)\ds.  
        \end{split}
    \end{equation}
    Finally, $I_4$ can be handled arguing as before. We use the continuity of $f'$ ensured by hypothesis~\ref{hyp:beta_pi}, the uniform boundedness of $\phi$ provided by \eqref{eq:solution_estimate}, and the uniform boundedness of the remainder $R_f$, as discussed above. Applying H\"older’s and Young’s inequalities and the continuous dependence estimate \eqref{eq:cont_dep}, we arrive at
    \begin{equation}\label{pier12}
        \begin{split}
            |I_4| &\leq C \intto \big( |\Phih| + |\phih - \phi|^2 + |\Thetah|\big)|-\Delta \Phih| \dx \ds\\
            & \leq C \intt \big( \norm{\Phih}_{\Hs} + \norm{\phih - \phi}_{\Hs} \norm{\phih - \phi}_{L^\infty(\Omega)} + \norm{\Thetah}_{\Hs}\big)\norm{-\Delta \Phih}_{\Hs} \ds\\
            & \leq \frac18 \intt \norm{-\Delta\Phih}_{\Hs}^2 \ds + \overline{C} \intt \big(\norm{\Phih}_{\Hs}^2 + \norm{\phih - \phi}_{\Hs}^2 \norm{\phih - \phi}_{\Ws}^2 + \norm{\Thetah}_{\Hs}^2\big) \ds\\
            & \leq \frac18 \intt \norm{-\Delta\Phih}_{\Hs}^2 \ds + \overline{C} \intt \big(\norm{\Phih}_{\Hs}^2+ \norm{\Thetah}_{\Hs}^2\big) \ds + C \norm{h}_{L^2(\Hs)}^4
        \end{split}
    \end{equation}
    for some suitable constant $\overline{C}$.
    The next step is to choose $v = \Pih$ in the equation \eqref{eq:diff_nutrient_weak} and integrate in time over  the time interval $(0,t)$, leading to
    \begin{align}
        &\frac{1}{2} \into |\Pih|^2 \dx \ds + \intto |\nabla \Pih|^2 \dx \ds + \lambdab \into |\Pih|^2 \dx \ds \notag \\
        &\quad = \chi \intto (- \Delta \Phih) \Pih \dx \ds 
        - \intto \lambdac \sigma\left[\h'(\phi)\Phih + R_{\h}(\phih - \phi)^2\right] \Pih \dx \ds \notag\\
       \begin{split}\label{eq:diff_est1_3}
            & \quad \quad - \intto \lambdac (\sigmah - \sigma)  \big(\h(\phih) - \h(\phi)\big) \Pih\dx \ds\\ 
            &\quad \quad - \intto |\Pih|^2 \big(\lambdac\h(\phi)  + \lambdad \k(\theta) \big) \dx \ds
       \end{split}\\
        & \quad \quad - \intto \lambdad \sigma \left[\k'(\theta)\Thetah + R_{\k}(\thetah - \theta)^2\right]\Pih \dx \ds \notag \\
        &\quad \quad - \intto \lambdad(\sigmah - \sigma)\big(\k(\thetah) - \k(\theta)\big) \Pih \dx \ds \eqqcolon \sum_{j=5}^{10} I_j.\notag
    \end{align}
    We next estimate the terms $I_5,\dots,I_{10}$.
    For $I_5$, we directly apply H\"older’s and Young’s inequalities with a small constant in front of the term with $(-\Delta \Phih )$, as before. This yields
    \begin{equation}
        \begin{split}
            |I_5| &\leq \chi \intto |-\Delta \Phih|\, |\Pih| \dx \ds \leq \chi \intt \norm{-\Delta \Phih}_{\Hs} \norm{\Pih}_{\Hs} \ds\\
            &\leq \frac18 \intt \norm{-\Delta \Phih}_{\Hs}^2 \ds + C \intt \norm{\Pih}_{\Hs}^2 \ds. 
        \end{split}
    \end{equation}
    The term $I_6$ is analogous to $I_1$, with the only difference that no interpolation inequality is needed. Thanks to \eqref{eq:solution_estimate} and \eqref{eq:cont_dep}, we have:
    \begin{equation}
        \begin{split}
            |I_6| &\leq C \intto |\sigma| \big(|\Phih| + |\phih - \phi|^2\big)|\Pih| \dx \ds\\
            &\leq C \intt \norm{\sigma}_{L^4(\Omega)} \big( \norm{\Phih}_{\Hs} + \norm{\phih - \phi}_{\Hs} \norm{\phih - \phi}_{L^{\infty}(\Omega)}\big) \norm{\Pih}_{L^4(\Omega)} \ds\\
            & \leq C \intt \big( \norm{\Phih}_{\Hs} + \norm{\phih - \phi}_{\Hs} \norm{\phih - \phi}_{\Ws}\big) \norm{\Pih}_{\Vs} \ds\\
            & \leq \frac18  \intt \norm{\Pih}_{\Vs}^2 \ds + C \intt \norm{\Phih}_{\Hs}^2 \ds + C \norm{h}_{L^2(\Hs)}^4.
        \end{split}
    \end{equation}
    For $I_7$, we exploit the Lipschitz continuity of $\h$, followed by H\"older’s and Young’s inequalities, as well as the continuous dependence estimate \eqref{eq:cont_dep} from \Cref{thm:wellposedness}. We infer that
    \begin{equation}
        \begin{split}
            |I_7| &\leq C \intto |\sigmah - \sigma|\, |\phih - \phi| \,|\Pih| \dx \ds\\
            &\leq C \intt \norm{\sigmah - \sigma}_{L^4(\Omega)} \norm{\phih - \phi}_{\Hs} \norm{\Pih}_{L^4(\Omega)} \ds\\
            &\leq C \norm{\phih - \phi}_{L^{\infty}(\Hs)} \intt \norm{\sigmah - \sigma}_{\Vs} \norm{\Pih}_{\Vs} \ds\\
            & \leq \frac18 \intt \norm{\Pih}_{\Vs}^2 \ds + C \norm{\phih - \phi}_{L^{\infty}(\Hs)}^2 \intt \norm{\sigmah - \sigma}_{\Vs}^2 \ds\\
            & \leq \frac18 \intt \norm{\Pih}_{\Vs}^2 \ds + C \norm{h}_{L^2(\Hs)}^4.
        \end{split}
    \end{equation}
    In $I_8$, we simply use the boundedness of $\h$ and $\k$  ensured by Hypothesis~\ref{hyp:nonlinearities}. We immediately deduce:
    \begin{equation}
        \begin{split}
            |I_8| \leq C \intt \norm{\Pih}_{\Hs}^2 \ds.
        \end{split}
    \end{equation}
    The term $I_9$ can be handled in a way similar to $I_6$, invoking also the continuous embedding $\Vs \hookrightarrow L^6(\Omega)$. We obtain
    \begin{equation}\label{pier13}
        \begin{split}
            |I_9| &\leq C \intto |\sigma| \big( |\Thetah| + |\thetah - \theta|^2\big) |\Pih| \dx \ds\\
            & \leq C \intt \norm{\sigma}_{L^4(\Omega)} \norm{\Thetah}_{\Hs} \norm{\Pih}_{L^4(\Omega)} \ds\\
            &\quad + C \intt \norm{\sigma}_{L^6(\Omega)}  \norm{\thetah - \theta}_{\Hs} 
            \norm{\thetah - \theta}_{L^6(\Omega)}\norm{\Pih}_{L^6(\Omega)} \ds\\
            & \leq C\intt \bigl( \norm{\Thetah}_{\Hs} + \norm{\thetah - \theta}_{\Hs} \norm{\thetah - \theta}_{\Vs} \bigr)\norm{\Pih}_{\Vs} \ds\\
            &\leq \frac18 \intt \norm{\Pih}_{\Vs}^2 \ds + \widehat{C\,}  \intt \norm{\Thetah}_{\Hs}^2 \ds + C \norm{h}_{L^2(\Hs)}^4
        \end{split}
    \end{equation}
    for a suitable constant $\widehat{C\,}$.
    Finally, we deal with $I_{10}$ as we did for $I_7$, finding
    \begin{equation}
        \begin{split}
            |I_{10}| &\leq C \intto |\sigmah - \sigma| \, |\thetah - \theta| \, |\Pih| \dx \ds\\
            &\leq C \intt \norm{\sigmah - \sigma}_{L^4(\Omega)} \norm{\thetah - \theta}_{\Hs} \norm{\Pih}_{L^4(\Omega)} \ds\\
            & \leq \frac18 \intt \norm{\Pih}_{\Vs}^2 \ds + C \norm{h}_{L^2(\Hs)}^4.
        \end{split}
    \end{equation}
    At this point, we sum the equations \eqref{eq:diff_est1_1}, \eqref{eq:diff_est1_2}, \eqref{eq:diff_est1_3}, taking into account the estimates we performed on $I_1, \dots, I_{10}$ to bind the resulting right-hand side. Then, we choose $R$ large enough so that
    \[ \frac R 2 > \overline C + \widehat{C\,}\]
    (please check \eqref{eq:diff_est1_1}, \eqref{pier11}, \eqref{pier12}, \eqref{pier13} once more).
    Thus,  renaming the constants and rearranging the terms, we have: 
    \begin{equation}
        \begin{split}
            &\into |\nabla (1 \ast_t \Thetah)|^2 \dx
            + \into |\Phih|^2 \dx + \tau \into |\nabla \Phih|^2 \dx 
            + \into |\Pih|^2 \dx\\
            & \quad \quad  + \intto |\Thetah|^2 \dx \ds + \intto |-\Delta \Phih|^2 \dx \ds + \intto |\nabla \Pih|^2 \dx \ds\\
            & \quad\leq C \biggl(\norm{h}_{L^2(\Hs)}^4 + \intt \norm{\Phih}_{\Hs}^2 \ds + \intt \norm{\Pih}_{\Hs}^2 \ds \biggr).
        \end{split}
    \end{equation}
    Then, invoking the Gronwall Lemma and recalling~\eqref{pier4}, we deduce that
    \begin{equation}\label{eq:diff_est1}
        \begin{split}
            &\norm{\Thetah}_{L^2(\Hs)} + \norm{1 \ast_{t} \Thetah}_{L^{\infty}(\Vs)} +  \norm{\Phih}_{L^{\infty}(\Hs) \cap L^2(\Ws)} + \tau^{1/2} \norm{\Phih}_{L^{\infty}(\Vs)}\\
            & \quad + \norm{\Pih}_{L^{\infty}(\Hs) \cap L^2(\Vs)} \leq C \norm{h}_{L^2(\Hs)}^2.
        \end{split} 
    \end{equation}
    \smallskip

\noindent \textbf{Second estimate.} We choose $v = \Thetah + \ell \Phih$ in the equation \eqref{eq:diff_temp_weak}, and integrate in time over the interval $(0,t)$, obtaining:
\begin{equation*}
    \begin{split}
        &\frac{1}{2} \into |\Thetah + \ell \Phih|^2 \dx + \intto |\nabla \Thetah|^2 \dx \ds \\
        &\quad= -\ell \intto \nabla \Thetah \cdot \nabla \Phih \dx \ds \leq \frac{1}{2} \intto |\nabla \Thetah|^2 \dx \ds + C \intto |\nabla \Phih|^2 \dx,
    \end{split}
\end{equation*}
whence it follows that the left-hand side is bounded by $C \norm{h}_{L^2(\Hs)}^4$ thanks to the control of $\norm{\nabla \Phih}_{L^{2}(\Hs)}$ in \eqref{eq:diff_est1}. Therefore, we have that
\begin{equation} \label{eq:diff_est2}
    \begin{split}
        \norm{\Thetah}_{L^{\infty}(\Hs)\cap L^2(\Vs)} &\leq C\bigl(\norm{\ell \Phih}_{L^{\infty}(\Hs)} + \norm{\Thetah + \ell \Phih}_{L^{\infty}(\Hs)}\bigr) + C\norm{\nabla \Thetah }_{L^2(\Hs)}\\
        &\leq C \norm{\Phih}_{L^{\infty}(\Hs) \cap L^2(\Vs}) \leq C \norm{h}_{L^2(\Hs)}^2.
    \end{split}
\end{equation}
\smallskip

\noindent \textbf{Third estimate.} We sum the equation \eqref{eq:diff_ch1_weak} tested with $v = \tau \Upsilonh$ and the equation \eqref{eq:diff_ch2_weak} tested with $v = - \Upsilonh$. Notice that we are going to keep track of $\tau$, whence it is clear that the following estimates hold both when the parameter is strictly positive and when it vanishes. Estimating the right-hand side as before, we have:
\begin{align*}
        &\tau \into |\nabla \Upsilonh|^2 \dx  + \into |\Upsilonh|^2 \dx \\
        &\quad = \tau \into (\lambdap \sigma - \lambdaa - \lambdae \theta) \left[\h'(\phi) \Phih + R_{\h} (\phih -\phi)^2\right] \Upsilonh \dx \\
        & \quad \quad + \tau \into \left[\lambdap(\sigmah - \sigma) - \lambdae(\thetah - \theta)\right]\left(\h(\phih)-\h(\phi)\right) \Upsilonh \dx \\
        &\quad \quad + \tau\into (\lambdap \Pih - \lambdae \Thetah) \h(\phi) \Upsilonh \dx \\
        &\quad \quad +\into \left[ -\Delta \Phih + f'(\phi)\Phih + R_{f} (\phih - \phi)^2 - \chi \Phih - \Lambda \Thetah \right] \Upsilonh \dx \\
        &\quad \leq C \tau \into \big(|\sigma| + 1\big) \big(|\Phih| + |\phih - \phi|^2\big) |\Upsilonh| \dx \\
        &\quad \quad + C \tau \into  \left[ \big(|\sigmah - \sigma| + |\thetah - \theta|\big)|\phih - \phi| + |\Pih| + |\Thetah| \right]  |\Upsilonh| \dx \\
        &\quad \quad+ C \into \left( |-\Delta \Phih| + |\Phih| + |\phih - \phi|^2 + |\Thetah|\right) |\Upsilonh| \dx .
    \end{align*}
    Then, we apply the H\"older and the Young inequalities with a small parameter $\epsilon > 0$ fixed but yet to be chosen, obtaining:
    \begin{equation*}
        \begin{split}
            &\tau \into |\nabla \Upsilonh|^2 \dx  + \into |\Upsilonh|^2 \dx \\
            &\quad \leq C \tau \Big[ \big(\norm{\sigma}_{L^{\infty}(L^6(\Omega))} + 1 \big)\big(\norm{\Phih}_{L^3(\Omega)} + \norm{\phih - \phi}_{L^6(\Omega)}^2\big)\Big]\norm{\Upsilonh}_{\Hs}\\
            &\quad \quad +C\tau \Big[\big(\norm{\sigmah - \sigma}_{\Hs} + \norm{\thetah - \theta}_{\Hs} \big) \norm{\phih - \phi}_{L^\infty (\Omega)} + \norm{\Pih}_{\Hs} + \norm{\Thetah}_{\Hs}\Big]\norm{\Upsilonh}_{\Hs}\\
            &\quad \quad + C \Big[\norm{-\Delta \Phih}_{\Hs} + \norm{\Phih}_{\Hs} + \norm{\phih - \phi}_{L^4(\Omega)}^2 + \norm{\Thetah}_{\Hs}\Big]\norm{\Upsilonh}_{\Hs}\\
            &\quad \leq \left(\epsilon\tau + \frac{1}{2}\right)\norm{\Upsilonh}_{\Hs}^2 + C_{\epsilon}\tau \left[\norm{\Phih}_{\Vs}^2 + \norm{\phih - \phi}_{\Vs}^4\right]\\
            &\quad\quad + C_{\epsilon}\tau \left[ \big(\norm{\sigmah - \sigma}_{L^{\infty}(\Hs)}^2 + \norm{\thetah - \theta}_{L^{\infty}(\Hs)}^2 \big) \norm{\phih - \phi}_{\Ws}^2 + \norm{\Pih}_{\Hs}^2 + \norm{\Thetah}_{\Hs}^2\right]\\
            &\quad\quad + C \Big[\norm{-\Delta \Phih}_{\Hs}^2 + \norm{\Phih}_{\Hs}^2 + \norm{\phih - \phi}_{\Vs}^4 + \norm{\Thetah}_{\Hs}^2\Big].
        \end{split}
    \end{equation*}
    If $\tau = 0$, the terms with factor $\tau$ disappear. If $\tau > 0$, we choose $\epsilon$ small enough (e.g., $\epsilon = 1/ (4\tau)$), and then we proceed in the same way. Integrating in time and exploiting the estimates \eqref{eq:cont_dep}, \eqref{eq:diff_est1}, \eqref{eq:diff_est2}, as well as 
    \begin{equation*}
        \norm{v}_{L^4(\Vs)} \leq C \norm{v}_{L^{\infty}(\Hs)}^{1/2} \norm{v}_{L^2(\Ws)}^{1/2}  \qquad \forall v \in L^{\infty}(0,T; \Hs) \cap L^2(0,T;\Ws)
    \end{equation*}
    derived straightforwardly from the Gagliardo--Nirenberg interpolation inequality (see, e.g., \cite{AdamsFournier2003}), we have:
    \begin{equation}\label{eq:diff_est3}
        \begin{split}
            \tau \norm{\nabla \Upsilonh}_{L^2(\Hs)}^2 + \norm{\Upsilonh}_{L^2(\Hs)}^2 \leq C (1+\tau) \norm{h}_{L^2(\Hs)}^4.
        \end{split}
    \end{equation}
    \noindent Collecting the estimates \eqref{eq:diff_est1}, \eqref{eq:diff_est2}, \eqref{eq:diff_est3}, the inequality \eqref{eq:diff_thesis} follows. Thus, the limit \eqref{eq:limit_diff} is proved, and the proof is complete.
\end{proof}

\section{First-order necessary conditions of optimality}
\label{sec:first}

In this section, we derive first-order necessary optimality conditions characterising the optimal controls whose existence has been established in \Cref{thm:existence_opt_control}. 
Let $u^* \in \Uad$ be an optimal control and let $u \in \Uad$ be arbitrary. Since the admissible set $\Uad$ is convex, for every $s \in (0,1)$ the convex combination $(1-s) u^* + s u$ still belongs to $\Uad$. Hence, by the optimality of $u^*$, we have
\begin{equation}\label{eq:opt_cond_0}
    \frac{\J(u^* + s (u - u^*)) - \J(u^*)}{s} \ge 0.
\end{equation}
Recall that $\J(u) = \JJ(\S_1(u), \S_2(u), u)$. The cost functional $\JJ$, regarded as a mapping from
\begin{equation*}
    \mathcal{Z} \coloneqq C^0([0,T];\Hs) \times C^0([0,T];\Hs) \times L^2(0,T;\Hs)
\end{equation*}
into $\RR$, is convex and Fréchet differentiable. Moreover, by \Cref{thm:diff_solution_op}, the solution operator $\S \colon \Um \to \mathcal{Y}$ is Fréchet differentiable. In particular, its components $\S_1$ and $\S_2$ are Fréchet differentiable as well. Owing to the continuous embeddings
\begin{alignat*}{2}
        &C^0([0,T]; \Hs) \cap L^2(0,T;\Vs) &\hookrightarrow C^0([0,T];\Hs),\\
        &C^0([0,T]; \Hs) \cap L^2(0,T;\Ws) &\hookrightarrow C^0([0,T];\Hs),
\end{alignat*}
the mappings $\S_1$ and $\S_2$ are in fact Fréchet differentiable also as operators from $\Um$ into $C^0([0,T];\Hs)$. Since the composition of Fréchet differentiable mappings is again Fréchet differentiable, it follows that $\J$ is Fréchet differentiable on $\Um$.
Hence, passing to the limit in \eqref{eq:opt_cond_0} as $s \to 0^+$, we obtain
\begin{equation*}
    D\J(u^*)[u - u^*] \geq 0
\end{equation*}
for all $u \in \Uad$.
By explicitly computing this derivative via the chain rule, we arrive at the following result, which provides a first version of the optimality necessary condition we are looking for.

\begin{prop}
    Let $u^* \in \Uad$ be an optimal control for the problem \eqref{eq:optimal_control} with associated optimal state $(\theta^*, \phi^*, \mu^*, \sigma^*)$. Then, for every $u \in \Uad$, the following inequality is satisfied 
    \begin{equation}\label{eq:opt_cond0}
        \begin{split}
            &b_1 \intTo (\theta^* - \theta_Q) \zeta \dx \dt + b_2\into (\theta^*(T) -\theta_{\Omega}) \zeta(T) \dx\\
            &\quad + b_3 \intTo (\phi^* - \phi_Q)\xi \dx \dt + b_4\into (\phi^*(T) -\phi_{\Omega}) \xi(T) \dx\\
            &\quad + b_5 \intTo u^*(u-u^*) \dx \dt \geq 0,
        \end{split}
    \end{equation}
    where $\zeta$ and $\xi$ are the first two components of the solution to the linearised system in $u^*$ (see \eqref{eq:lin_problem_integrals}) for $h = u - u^*$.
\end{prop}

\noindent This characterisation is not yet fully satisfactory, as it requires solving infinitely many linearised systems, one for each $u \in \Uad$.
To reformulate condition \eqref{eq:opt_cond0} removing the linearised variables,  we introduce the adjoint system \eqref{eq:adj_system}--\eqref{eq:adj_final_cond}.

\subsection{The adjoint system}
We denote by $(\z, \p, \q, \r)$ the adjoint state variables, which are associated with a control $u^* \in \Uad$ and the corresponding state $(\theta^*, \phi^*, \mu^*, \sigma^*)$. The formal equations and conditions for the backward-in-time adjoint system are the following:
\begin{subequations}\label{eq:adj_system}
    \begin{align}
        &-\partial_t \z - \Delta \z = \Lambda \q - \lambdae \h(\phi^*) \p - \lambdad \sigma^*\k'(\theta^*)\r + b_1 (\theta^* - \theta_Q),\label{eq:adj_temp}\\
        &\begin{aligned}
             {}-\partial_t \left(\p + \tau \q  + \ell \z\right) - \Delta \left(\q - \chi \r\right) = & \ (\lambdap \sigma^* - \lambdaa - \lambdae \theta^*)\h'(\phi^*)\p - f'(\phi^*)\q\\
             &  - \lambdac \sigma^* \h'(\phi^*)\r + b_3(\phi^* - \phi_Q),
        \end{aligned}\label{eq:adj_ch1}\\
        &-\Delta \p = \q, \label{eq:adj_ch2}\\
        &-\partial_t \r - \Delta \r = \lambdap  \h(\phi^*) \p + \chi \q- \lambdac \h(\phi^*)\r - \lambdab \r - \lambdad \k(\theta^*)\r, \label{eq:adj_nutrient}
    \end{align}
\end{subequations}
posed in $Q$, and complemented with the homogeneous Neumann boundary conditions
\begin{equation}\label{eq:adj_bound_cond}
    \partial_{\nnu} \z = \partial_{\nnu} \p = \partial_{\nnu} \q = \partial_{\nnu} \r = 0, 
\end{equation}
set on $\Sigma$, together with the final conditions
\begin{subequations}\label{eq:adj_final_cond}
    \begin{align}
        &\z(T) = b_2(\theta^*(T) - \theta_{\Omega}), \label{pier-f1}\\
        &\p(T) + \tau\q(T) = b_4 (\phi^*(T) - \phi_{\Omega}) -  \ell b_2 (\theta^*(T) - \theta_{\Omega}),\label{pier-f2}\\
        &\r(T) = 0.\label{pier-f3}
    \end{align}
\end{subequations}

\begin{thm}
\label{prop:adj_wellposedness}
    Assume that the hypotheses \ref{hyp:constants}--\ref{hyp:cost-fun} hold and let $\tau$ be greater than or equal to $0$. In addition, fix a control $u^* \in \Uad$  with corresponding state $(z^*, \phi^*, \mu^*, \sigma^*)$. Then, the adjoint system \eqref{eq:adj_system}--\eqref{eq:adj_final_cond} has a unique weak solution $(\z, \p, \q, \r)$ such that
    \begin{gather*}
        \z \in H^1(0,T; \Hs) \cap C^0([0,T];\Vs) \cap L^2(0,T; \Ws)\\
        \p + \tau \q \in H^1(0,T; \Vs')\cap C^0([0,T]; \Hs), \quad  \q \in L^2(0,T; \Vs),\\
        \p \in L^\infty(0,T; \Vs) \cap L^2(0,T;\Ws\cap H^3(\Omega)),\\
        \r \in H^1(0,T; \Hs) \cap C^0([0,T];\Vs) \cap L^2(0,T; \Ws),
    \end{gather*}
    which satisfies the final conditions~\eqref{eq:adj_final_cond} a.e.~in $\Omega$ and fulfills \eqref{eq:adj_system}--\eqref{eq:adj_bound_cond} in variational sense, i.e., it satisfies
    \begin{subequations}
    \small
        \begin{align}
            &\begin{aligned}
                 &{}-\duality{\partial_t \z}{v}_{\Vs} + \into \nabla \z \cdot \nabla v \dx\\
                 &\quad = \into \big[\Lambda \q - \lambdae \h(\phi^*) \p - \lambdad \sigma^*\k'(\theta^*)\r + b_1 (\theta^* - \theta_Q)\big] v\dx ,
            \end{aligned}
            \label{eq:adj_var_temp}\\
            &\begin{aligned}
                 &{}-\duality{\partial_t (\p + \tau \q  + \ell \z)}{v}_{\Vs}  + \into \nabla(\q - \chi \r )\cdot \nabla v \dx
                 = \into (\lambdap \sigma^* - \lambdaa - \lambdae \theta^*)\h'(\phi^*)\p v
                 \\
                 &\quad{} + \into \bigl[ - f'(\phi^*)\q - \lambdac \sigma^* \h'(\phi^*)\r + b_3(\phi^* - \phi_Q)\bigr] v \dx,
            \end{aligned}\label{eq:adj_var_ch1}\\
            & \into \nabla \p \cdot \nabla v \dx = \into \q v \dx, \label{eq:adj_var_ch2}\\
            &\begin{aligned}
            &-\duality{\partial_t \r}{v}_{\Vs} + \into \nabla \r \cdot \nabla v \dx \\
            &\quad{} = \into \big[\lambdap  \h(\phi^*) \p + \chi \q- \lambdac \h(\phi^*)\r - \lambdab \r - \lambdad \k(\theta^*)\r\big] v \dx ,     
            \end{aligned}
            \label{eq:adj_var_nutrient}
        \end{align}
    \end{subequations}
    for all  $v\in\Vs$, a.e.~in $(0,T)$.
\end{thm}

\begin{proof}
First of all, we observe that the backward in time system \eqref{eq:adj_system}--\eqref{eq:adj_final_cond} is linear. Therefore, once the existence of solutions has been established by means of suitable energy estimates (cf. the subsequent estimate~\eqref{eq:adj_est_3}), uniqueness follows immediately by linearity.
Concerning existence, a rigorous argument would require a suitable approximation procedure. In particular, one may discretise the system in space by means of the Faedo--Galerkin scheme based on eigenfunctions of the operator $-\Delta$ with homogeneous Neumann boundary conditions, then prove the existence of discrete solutions by Carathéodory's theorem, derive a priori estimates that are uniform with respect to the discretisation parameter, and finally pass to the limit by compactness arguments. Since this procedure is standard in the analysis of linear parabolic systems, we limit ourselves to formally deriving the necessary a priori estimates at the continuous level, which can be justified within the aforementioned approximation framework.
Let $R > 0$ be a fixed (large) constant, whose value will be chosen later. 
We multiply equation \eqref{eq:adj_var_temp} by $\ell/\Lambda $ and test it with $v = \z - \Delta \z $, then we test equation~\eqref{eq:adj_var_ch1} with $v = \p + \tau \q + \ell \z$ and equation~\eqref{eq:adj_var_nutrient} with $v = R \r $. Summing the resulting equalities, integrating by parts in some term, and using \eqref{eq:adj_var_ch2} tested by $v=\q$,  we obtain
{\allowdisplaybreaks
\begin{align*}
    & {}-\frac{\ell}{2\Lambda} \difft \norm{\z}_{\Vs}^2 +\frac{\ell}{\Lambda} \into \bigl( |\nabla \z|^2 + |\Delta \z|^2\bigr)\dx - \frac{1}{2} \difft \into |\p + \tau \q + \ell \z|^2 \dx \\
    &\quad \quad + \into |\q|^2 \dx+ \tau \into |\nabla \q|^2 \dx -\frac{R}{2} \difft  \into |\r|^2 \dx+ R \into |\nabla \r|^2 \dx \\
    &\quad = \into \biggl[\ell \q (\z - \Delta \z) + \frac{\ell}{\Lambda} \bigl(- \lambda_E \h(\phi^*)\p - \lambda_D \sigma^* \k'(\theta^*)\r + b_1(\theta^* - \theta_Q) \bigr) (\z - \Delta \z)\biggr] \dx\\
    &\quad \quad {}+ \into \bigl[ \ell\, \q \, \Delta \z + \chi \nabla \r \cdot ( \nabla \p + \tau \nabla \q + \ell \,\nabla \z ) \bigr]\dx\\
    &\quad \quad + \into \Bigl\{\left[ (\lambda_P \sigma^* - \lambda_A - \lambda_E \theta^*)\h'(\phi^*)\p\right](\p + \tau \q + \ell \z) \\
    &\quad \qquad\qquad {}+ \left[ - f'(\phi^*)\q - \lambda_C \sigma^* \h'(\phi^*)\r 
    + b_3(\phi^* - \phi_Q) \right](\p + \tau \q + \ell \z)\Bigr\} \dx\\
    &\quad \quad + R \into \Bigl[ \lambda_P \h(\phi^*)\p + \chi \q - \lambda_C \h(\phi^*)\r - \lambda_D \k(\theta^*)\r -\lambdab \r \Bigr] \r \dx.
\end{align*}}
Now, we note that two terms on the \rhs, the ones involving $\q\, \Delta \z$, cancel each other. Then, we integrate the above identity with respect to time over the interval $(t,T)$ and exploit the final conditions \eqref{eq:adj_final_cond}. This yields
\begin{equation}\label{eq:adj_est_10}
    \begin{split}
        &\frac{1}{2} \left[ \frac{\ell}{\Lambda}\norm{\z}_{\Vs}^2 + \into |\p + \tau \q + \ell \z|^2 \dx + R \into |\r|^2 \dx \right] + \frac{\ell}{\Lambda} \inttTo \bigl( |\nabla \z|^2 + |\Delta \z|^2\bigr) \dx \ds \\
        &\quad \quad + \inttTo |\q|^2 \dx \ds + \tau \inttTo |\nabla \q|^2 \dx \ds + R \inttTo  |\nabla \r|^2 \dx \ds \\
        &\quad= \frac{1}{2} \left[\frac{\ell}{\Lambda}  b_2^2 \, \norm{\theta^*(T) - \theta_{\Omega}}^2_{\Vs} + b_4^2 \into |\phi^*(T) - \phi_{\Omega}|^2 \dx \right] + I_1 + \dots + I_4.
    \end{split}
\end{equation}
In what follows, we estimate the terms $I_1, \dots, I_4$, keeping explicit track of the parameter $\tau$. 
A recurrent argument consists in adding and subtracting the quantity $\tau \q + \ell \z$ whenever the variable $\p$ appears alone, so as to recover the combination $\p + \tau \q + \ell \z$ and pair it with the corresponding term on the left-hand side.
Concerning the term $I_1$, we write~it~as
\begin{equation*}
    \begin{split}
        I_1 &= \inttTo\biggl[  \ell \q \z + \frac{\ell}{\Lambda}\bigl(- \lambda_E \h(\phi^*)(\p + \tau \q + \ell \z) \bigr)  (\z - \Delta \z)\biggr]  \dx\ds\\
        &\quad + \tau \frac{\ell}{\Lambda} \inttTo\lambda_E \h(\phi^*) \q  (\z - \Delta \z) \dx\ds  + \frac{\ell^2}{\Lambda} \inttTo\lambda_E \h(\phi^*)\z (\z - \Delta \z)  \dx\ds\\
        &\quad + \frac{\ell}{\Lambda} \inttTo\left[ - \lambdad \sigma^* \k'(\theta^*)\r + b_1(\theta - \theta_Q) \right] (\z - \Delta \z) \dx\ds.
    \end{split}
\end{equation*}
In order to estimate $I_1$, we exploit the boundedness of $\h$ and $\k'$, ensured by assumption~\ref{hyp:nonlinearities}, and apply H\"older’s and Young’s inequalities. Owing to the continuous embedding $\Vs \hookrightarrow L^4(\Omega)$ and to the estimate~\eqref{eq:solution_estimate}, we first handle the second integral on the \rhs\ by splitting it into three terms:  
\begin{align*}
    &\tau \frac{\ell}{\Lambda} \inttTo\lambda_E \h(\phi^*) \q \z   \dx\ds
    + \tau \frac{\ell}{\Lambda} \inttTo\lambda_E \h'(\phi^*) \q \nabla \phi^* \cdot  \nabla \z \dx\ds \\ 
    &\quad \quad + \tau \frac{\ell}{\Lambda} \inttTo\lambda_E \h(\phi^*) \nabla \q \cdot  \nabla \z \dx\ds \\
    &\quad \leq \frac1{32} \inttTo |\q|^2 \dx \ds + C \tau^2 \inttTo|\z|^2 \dx\ds \\
    &\quad\quad + C \tau \inttT \norm{\q}_{L^4(\Omega)} \norm{\phi^*}_{L^4(\Omega)} \norm{\nabla \z}_H\ds + C \tau \inttT \norm{\nabla \q}_{H} \norm{\nabla \z}_H\ds \\
    &\quad\leq \frac1{16} \inttTo |\q|^2 \dx \ds + \frac\tau 8 \inttTo |\nabla \q|^2 \dx \ds
    + C \inttT \bigl(1+\tau^2\bigr)  \norm{\z}_V^2\ds.   
\end{align*}
At this point, using the continuity of the  embedding $\Ws \hookrightarrow L^\infty (\Omega)$ as well, it is not difficult to check that
\begin{equation}\label{eq:adj_est_11}
    \begin{split}
        I_1 &\leq \frac18 \inttTo |\q|^2 \dx \ds + \frac\tau 8 \inttTo |\nabla \q|^2 \dx \ds+ \frac{\ell}{2\Lambda} \inttTo |\Delta \z|^2 \dx \ds \\
       &\quad   + C\inttTo |\p + \tau \q + \ell \z|^2 \dx \ds  + C \inttTo |\r|^2 \dx\ds\\
        &\quad  + C \inttTo |\theta^* - \theta_Q|^2 \dx\ds + C \inttT \bigl(1+ \tau^2 + \norm{\sigma^*}_{\Ws}^2 \bigr)  \norm{\z}_V^2\ds.
    \end{split}
\end{equation}
We now turn to the estimate of $I_2 \displaystyle =\inttTo  \chi \nabla \r \cdot ( \nabla \p + \tau \nabla \q + \ell \,\nabla \z )  \dx \ds$. First, we observe that
\begin{equation*}
    \chi \inttTo \nabla \r \cdot \nabla \p \dx \ds = \chi \inttTo \q \r \dx \ds,
\end{equation*}
thanks to \eqref{eq:adj_var_ch2} tested with $\r$ and integrated in time. Once again, by H\"older’s and Young’s inequalities, we infer that 
\begin{equation}\label{eq:adj_est_12}
    \begin{split}
    I_2 &\leq \frac18 \inttTo |\q|^2 \dx \ds + 2 \chi^2 \inttTo |\r|^2 \dx \ds  + \frac{\tau}{8} \inttTo |\nabla \q|^2 \dx \ds \\
    &\quad  + 2 \tau \chi^2 \inttTo |\nabla \r|^2 \dx\ds + \frac{R}{2} \inttTo |\nabla \r|^2 \dx \ds + \frac{(\chi \ell)^2}{2R} \inttTo |\nabla \z|^2 \dx \ds.
    \end{split}
\end{equation}%
We now estimate the term $I_3$. As in the previous cases, we first rewrite the terms involving $\p$ so as to make the combination $\p + \tau \q + \ell \z$ appear explicitly. We then exploit the boundedness of $\h'$, ensured by assumption \ref{hyp:nonlinearities}, the continuity of $f'$ given by assumption \ref{hyp:beta_pi}, as well as the uniform boundedness of $\theta^*$ and $\phi^*$ provided by \Cref{thm:wellposedness}. By H\"older’s and Young’s inequalities, we obtain
\begin{equation}\label{eq:adj_est_13}
    \begin{split}
        I_3 &= \inttTo (\lambda_P \sigma^* - \lambda_A - \lambda_E \theta^*)\h'(\phi^*) |\p + \tau \q + \ell \z|^2 \dx \ds \\
        &\quad - \inttTo (\lambda_P \sigma^* - \lambda_A - \lambda_E \theta^*) (\tau \q + \ell \z)(\p + \tau \q + \ell \z) \dx \ds \\
        &\quad + \inttTo \left[ -f'(\phi^*)\q - \lambda_C \sigma^* \h'(\phi^*)\r + b_3(\phi^* - \phi_Q) \right] (\p + \tau \q + \ell \z) \dx \ds \\
        &\leq C\bigl(1 + \tau^2\bigr)\inttT \bigl(\norm{\sigma^*}_{\Ws}^2+ 1\bigr) \into |\p + \tau \q + \ell \z|^2 \dx \ds + \frac18 \inttTo |\q|^2 \dx \ds \\
        &\quad + C \inttTo |\z|^2 \dx \ds + C\inttTo |\r|^2 \dx \ds + C \inttTo |\phi^* - \phi_Q|^2 \dx\ds.
    \end{split}
\end{equation}
Finally, we analyse the term $I_4$. We rewrite the first contribution as before, so as to make the quantity $\p + \tau \q + \ell \z$ appear. Moreover, we observe that the last three terms inside the brackets, once multiplied by $\r$, are nonpositive and can therefore be neglected in the inequality. By arguments analogous to those used above, we deduce that
\begin{equation}\label{eq:adj_est_14}
    \begin{split}
        I_4 &\leq R \inttTo \left[ \lambda_P \h(\phi^*)(\p + \tau \q + \ell \z) - \lambda_P \h(\phi^*)(\tau \q + \ell \z) + \chi \q \right] \r \dx \ds\\
        &\leq C R \inttTo |\p + \tau \q + \ell \z|^2 \dx \ds 
        + CR\bigl(1 + \tau^2R +R\bigr)\inttTo |\r|^2 \dx \ds \\
        &\quad + CR\inttTo |\z|^2 \dx \ds  + \frac18 \inttTo |\q|^2 \dx \ds.
    \end{split}
\end{equation}%
We now return to equality \eqref{eq:adj_est_10} and make use of the estimates \eqref{eq:adj_est_11}--\eqref{eq:adj_est_14}. Rearranging the terms and using the estimate~\eqref{eq:solution_estimate} and the assumptions in \ref{hyp:cost-fun}, we obtain
{%
\allowdisplaybreaks
\begin{align*}
    &\frac{1}{2} \left[ \frac{\ell}{\Lambda}\norm{\z}_{\Vs}^2 + \into |\p + \tau \q + \ell \z|^2 \dx + R \into |\r|^2 \dx \right] + \frac{\ell}{2\Lambda} \inttTo \bigl( |\nabla \z|^2 + |\Delta \z|^2\bigr)\dx \ds \\
    &\quad \quad + \frac12 \inttTo |\q|^2 \dx \ds + \frac\tau 2\inttTo |\nabla \q|^2 \dx \ds + \Bigl(\frac R 2 - 2 \tau \chi^2  \Bigr)\inttTo  |\nabla \r|^2 \dx \ds \\
    &\quad{}\leq C 
    + C\inttT \bigl(1+ \tau^2 \bigr)\bigl(1 + \norm{\sigma^*}_{\Ws}^2 \bigr) \into |\p + \tau \q + \ell \z|^2 \dx \ds 
    + C \inttTo |\r|^2 \dx\ds \\
    &\quad\quad{} + C \inttT \bigl(1+ \tau^2 \bigr)\bigl(1 + \norm{\sigma^*}_{\Ws}^2 \bigr)  \norm{\z}_V^2\ds + C R \inttTo |\p + \tau \q + \ell \z|^2 \dx \ds \\
    &\quad\quad{} + CR\bigl(1 + \tau^2R +R\bigr)\inttTo |\r|^2 \dx \ds 
    + CR\inttTo |\z|^2 \dx \ds . 
\end{align*}%
}%
We now fix $R $ sufficiently large, namely $R > 4 \tau \chi^2 $, in order that the coefficient of the last term on the \lhs\ be positive.
Then, we can apply Gronwall’s lemma, taking into account that $\sigma^*$ is uniformly bounded in $L^2(0,T; \Ws)$ by a constant depending only on $M$, as ensured by \Cref{thm:wellposedness}. In view of~\eqref{pier4} we find out that
\begin{equation}\label{eq:adj_est_3}
    \begin{split}
        &\norm{\z}_{L^{\infty}(\Vs) \cap L^2(\Ws)} + \norm{\p + \tau \q + \ell \z}_{L^{\infty}(\Hs)} \\
        &\quad{}+ \norm{\q}_{L^2(\Hs)} + \tau^{1/2} \norm{\q}_{L^2(\Vs)} + \norm{\r}_{L^{\infty}(\Hs) \cap L^2(\Vs)}  \leq C_{M}.
    \end{split}
\end{equation}
Hence, we can also infer that 
\begin{equation}\label{eq:adj_est_4}
    \norm{\p}_{L^2(\Hs)} \leq C \left(\norm{\p + \tau \q + \ell \z}_{L^{\infty}(\Hs)} + \tau \norm{\q}_{L^2(\Hs)} + \ell \norm{\z}_{L^\infty(\Hs)}\right) \leq C_{M}.
\end{equation}
Then, a careful inspection of the terms in \eqref{eq:adj_var_temp} 
(cf. also \eqref{eq:adj_temp}), along with 
the bounds in \eqref{eq:adj_est_3}, shows that $-\partial_t \z = \Delta \z + \Lambda \q - \lambdae \h(\phi^*) \p - \lambdad \sigma^*\k'(\theta^*)\r + b_1 (\theta^* - \theta_Q)$ is bounded in $L^2(0,T;\Hs)$, whence 
\begin{equation} \label{pier10}
    \norm{\z}_{H^1(\Hs)\cap L^{\infty}(\Vs) \cap L^2(\Ws)} \leq C_{M}.
\end{equation}
Also, we note that the \rhs\ of equation~\eqref{eq:adj_var_nutrient} (see~\eqref{eq:adj_nutrient} as well) is bounded in $L^2(0,T;\Hs)$ and the final value $\r(T)=0$ is fixed in $\Vs$. 
Therefore, one can apply a well-known parabolic regularity estimate (see, e.g., \cite[Chapter~3]{Lions_Magenes_12}) and deduce that 
\begin{equation}\label{pier14}
    \norm{\r}_{H^1(\Hs)\cap L^{\infty}(\Vs) \cap L^2(\Ws)} \leq C_{M}.
\end{equation}
Moreover, in the light of \eqref{eq:adj_est_3} and \eqref{eq:adj_est_4}, by \eqref{pier4} and elliptic regularity applied to the equation \eqref{eq:adj_var_ch2} (or \eqref{eq:adj_ch2}), it follows that
\begin{equation*}
    \norm{\p}_{L^2(\Ws)} + \tau^{1/2} \norm{\p}_{L^2(H^3(\Omega))}  \leq C_{M}.
\end{equation*}
At this point, we can rewrite~\eqref{eq:adj_var_ch1} as
\begin{align}\label{pier16}
        &{}-\duality{\partial_t (\p + \tau \q)}{v}_{\Vs} + \into \nabla\q \cdot \nabla v \dx
        = \into g v              
\end{align}
for all $v\in \Vs$, a.e.~in $(0,T)$, where 
\begin{align*}
&g = \ell \partial_t \z +\chi \Delta \r + (\lambdap \sigma^* - \lambdaa - \lambdae \theta^*)\h'(\phi^*)\p  \\
&\quad \quad - f'(\phi^*)\q - \lambdac \sigma^* \h'(\phi^*)\r + b_3(\phi^* - \phi_Q) 
 \end{align*}
is already known to be bounded in $L^2(0,T;\Hs)$ (cf.~\eqref{pier10} and \eqref{pier14}). Now, we aim to take $v=\q$ in \eqref{pier16} and integrate in time from $t$ to $T$. With the help of \eqref{eq:adj_var_ch2} we obtain
\begin{equation}\label{pier17}
    \begin{split}
        &\frac{1}{2} \left[ \into |\nabla \p |^2 \dx + \tau \into |\q|^2 \dx \right] +  \inttTo |\nabla \q|^2  \dx \ds \\
        &\quad\leq  \frac{1}{2} \left[ \into |\nabla \p (T) |^2 \dx + \tau \into |\q (T)|^2 \dx \right] + \norm{g}_{L^2(\Hs)} \norm{\q}_{L^2(\Hs)}.
    \end{split}
\end{equation}
In order to bound the first part of the \rhs\ we test \eqref{pier-f2} by $ \q(T) = - \Delta \p(T)$ (see~\eqref{eq:adj_var_ch2}) and integrate by parts so that
\begin{equation}\label{pier18}
    \begin{split}
&\into |\nabla \p (T) |^2 \dx + \tau \into |\q (T)|^2 \dx \\
&\quad =  \into \nabla \bigl( b_4 (\phi^*(T) - \phi_{\Omega}) -  \ell b_2 (\theta^*(T) - \theta_{\Omega}) \bigr) \cdot \nabla \p(T) \dx \\
        &\quad\leq   \frac{1}{2} \into \big|\nabla \bigl( b_4 (\phi^*(T) - \phi_{\Omega}) -  \ell b_2 (\theta^*(T) - \theta_{\Omega}) \bigr)  \big|^2 \dx
        + \frac{1}{2} \into |\nabla \p (T) |^2 \dx  .     
   \end{split}
\end{equation}
Then, as $ b_4 (\phi^*(T) - \phi_{\Omega}) -  \ell b_2 (\theta^*(T) - \theta_{\Omega})$ is bounded in $V$ due to \eqref{eq:solution_estimate} and \ref{hyp:cost-fun}, the \rhs\ of \eqref{pier17} is bounded as well. Hence, in view of \eqref{eq:adj_est_3} we deduce that
\begin{equation}\label{pier19}
    \begin{split}
        & \norm{\nabla \p}_{L^{\infty}(\Hs)} + 
         \tau^{1/2} \norm{\q}_{L^\infty (\Hs)} + \norm{\q}_{ L^2(\Vs)}  \leq C_{M}.
    \end{split}
\end{equation}
Thanks to~\eqref{pier19}, the estimate \eqref{eq:adj_est_4} can be improved to $L^\infty (0,T;\Hs)$ for $\p$, so that, also by \eqref{eq:adj_ch2} and elliptic regularity, 
\begin{equation*}
    \begin{split}
        & \norm{\p}_{L^{\infty}(\Vs)} + 
         \tau^{1/2} \norm{\p}_{L^\infty (\Ws)} + \norm{\p}_{ L^2(H^3(\Omega))}  \leq C_{M}.
    \end{split}
\end{equation*}
Finally, by comparison in the equation \eqref{pier16} we infer that
\begin{equation*}
   \norm{\partial_t (\p + \tau \q)}_{L^2(\Vs')} \leq C_{M}.
\end{equation*}
In particular, note that $\p+\tau \q$ is continuous from $[0,T]$ to $\Hs$ as it belongs to $H^1 (0,T;\Vs') \cap L^2 (0,T;\Vs)$.
\end{proof}

\subsection{Conclusions}

We now rewrite the optimality conditions in terms of the solution to the adjoint problem.
\begin{thm}[First-order necessary conditions of optimality]
Assume that the hypotheses \ref{hyp:constants}--\ref{hyp:cost-fun} hold and let $\tau$ be greater than or equal to $0$. Let $u^* \in \Uad$ be an optimal control with associated optimal state $(\theta^*, \phi^*, \mu^*, \sigma^*)$. Then, for every $u \in \Uad$ it holds that
    \begin{equation}\label{eq:opt_cond}
        \intTo (\z + b_5 u^*)(u-u^*) \dx \dt \geq 0,
    \end{equation}
    where $\z$ is the first component of the solution $(\z, \p, \q, \r)$ to
    the adjoint system \eqref{eq:adj_system}--\eqref{eq:adj_final_cond} related to $u^*$.
\end{thm}

\begin{proof}
     First, we recall the linearised equations~\eqref{eq:lin_problem_integrals} in which the increment $h$ in \eqref{eq:lin_temp_weak} is just $u-u^*$.
     Then, we proceed by testing \eqref{eq:lin_temp_weak} with $\z$, \eqref{eq:lin_ch1_weak} with $\p$, \eqref{eq:lin_ch2_weak} with $\q$, and \eqref{eq:lin_nutrient_weak} with $\r$, and summing the resulting identities. Next, we test the equations~\eqref{eq:adj_var_temp}, \eqref{eq:adj_var_ch1}, \eqref{eq:adj_var_ch2}, and \eqref{eq:adj_var_nutrient} with $\zeta$, $\xi$, $\eta$, and $\rho$, respectively, and subtract these from the previous sum to obtain, after some cancellations, the following equality:
    \begin{equation*}
        \begin{split}
            &\difft \left[ \into \left(\zeta\z + \ell \xi \z + \xi \p + \tau \xi \q + \rho \r \right)\dx \right]\\
            &\quad = \into (u-u^*)\z \dx - b_1 \into (\theta^* - \theta_Q) \zeta \dx - b_3 \into (\phi^* - \phi_Q)\xi \dx.
        \end{split}
    \end{equation*}
    Integrating in time over the interval $(0,T)$ and exploiting the initial/final conditions \eqref{eq:lin_init_cond}, \eqref{eq:adj_final_cond} leads to:
    \begin{equation}\label{eq:opt_cond_2}
        \begin{split}
            &b_2\into (\theta^*(T) -\theta_{\Omega}) \zeta(T) \dx + b_4\into (\phi^*(T) -\phi_{\Omega}) \xi(T) \dx\\
            &\quad = \intTo (u-u^*)\z \dx \dt - b_1 \intTo (\theta^* - \theta_Q) \zeta \dx \dt - b_3 \intTo (\phi^* - \phi_Q)\xi \dx \dt.
        \end{split}
    \end{equation}
    Finally, combining \eqref{eq:opt_cond_2} with the previously obtained \eqref{eq:opt_cond0}, the thesis follows.
\end{proof}

\begin{remark}
The inequality~\eqref{eq:opt_cond}, which holds for all $u \in \Uad$, shows that $u^*$ coincides with the $L^2(Q)$-projection of $-\z/b_5$ onto the closed and convex set $\Uad$.
\end{remark}

\section*{Acknowledgments}
The authors gratefully acknowledge partial support from the Next Generation EU Project No.~P2022Z7ZAJ (A Unitary Mathematical Framework for Modelling Muscular Dystrophies). 
G.C. is funded by the Deutsche Forschungsgemeinschaft (DFG, German Research Foundation) under Germany's Excellence Strategy – The Berlin Mathematics Research Center MATH+ (EXC-2046/2, project ID: 390685689).
All authors are members of GNAMPA (Gruppo Nazionale per l’Analisi Matematica, la Probabilità e le loro Applicazioni), part of INdAM (Istituto Nazionale di Alta Matematica).

\printbibliography

\end{document}